\documentclass[10pt]{article}

\usepackage{latexsym}
\usepackage{amsfonts,amssymb}

\usepackage{commands}

\usepackage{hyperref}

\numberwithin{equation}{section}

\begin{document}

\title{Infinite dimensional Ornstein-Uhlenbeck processes with unbounded diffusion 
-- Approximation, quadratic variation, and It\^o formula} 
\par
\author{John Karlsson and J\"org-Uwe L\"obus
\\ Matematiska institutionen \\ 
Link\"opings universitet \\ 
SE-581 83 Link\"oping \\ 
Sverige 
}
\date{}
\maketitle
{\footnotesize
\noindent
\begin{quote}
{\bf Abstract}
The paper studies a class of Ornstein-Uhlenbeck processes on the classical Wiener space. These  
processes are associated with a diffusion type Dirichlet form whose corresponding diffusion 
operator is unbounded in the Cameron-Martin space. It is shown that the distributions of certain 
finite dimensional Ornstein-Uhlenbeck processes converge weakly to the distribution of such an 
infinite dimensional Ornstein-Uhlenbeck process. For the infinite dimensional processes, the 
ordinary scalar quadratic variation is calculated. Moreover, relative to the stochastic calculus 
via regularization, the scalar as well as the tensor quadratic variation are derived. A 
related It\^o formula is presented. 

\noindent

{\bf AMS subject classification (2010)} primary 60J60 secondary 60G15

\noindent
{\bf Keywords} Infinite dimensional Ornstein-Uhlenbeck process, quadratic variation, 
It\^o formula, weak approximation.
\end{quote}
}

\section{Introduction} 
\label{intro:section:1} 

Over the past two decades, infinite dimensional diffusion processes have become a central 
focus of stochastic analysis. One important class of infinite dimensional stochastic processes 
is the set of Ornstein-Uhlenbeck type processes. 

We wish to investigate a fairly accessible representative of the relatively abstract class 
of infinite dimensional stochastic processes with unbounded diffusion introduced in 
\cite{Lobus2004}, \cite{WangWu2008}, \cite{ChenWu2014}, and \cite{KarlssonLobus2014}. 
Existence and representation of standard elements of the stochastic calculus such as 
quadratic variation and It\^o formula may convince that these processes fit in the general 
concept of infinite dimensional stochastic processes. We would also like to emphasize that 
the just mentioned references deal with a class of stochastic processes taking values in 
certain path spaces, i.e. in Banach spaces.  

As a result of this paper, we have scalar as well as tensor quadratic variation and 
the corresponding It\^o formula available for a class of infinite dimensional 
Ornstein-Uhlenbeck processes 
with unbounded diffusion. For this we have used the recently developed stochastic 
calculus via regularization, see \cite{RussoVallois2007}, \cite{Girolami2014}, and 
\cite{GirolamiRusso2014}. Even if the diffusion of such infinite dimensional 
Ornstein-Uhlenbeck processes is governed by unbounded operators these processes 
can be weakly approximated by finite dimensional Ornstein-Uhlenbeck processes.
\medskip 

Let $C_0([0,1];\R^d)$ denote the set of all continuous $\R^d$-valued functions $\gamma$ on 
$[0,1]$ with $\gamma(0)=0$ and let $C_0([0,1];\R^d)$ be endowed with the sup-norm. Introduce 
the following set of cylindrical functions over $\Omega\equiv C_0([0,1];\R^d)$
\begin{align*}
	Y:=&\left\{F(\gamma)=f\left(\gamma(s_1),\dots,\gamma(s_k)\right), \gamma \in \Omega:  
	0<s_1<\dots<s_k=1,\ \vphantom{\dot{f}}\right.\\
	&\qquad\left. s_1,\dots,s_k\in \left\{\textstyle{\frac{l}{2^n}}:l \in \{1,\dots,2^n\} 
	\right\},\ n\in \N,\  f\in C^{\infty}_p(({\RE^d})^k), k\in \N\vphantom{\dot{f}}\right\}
\end{align*} 
where $C^{\infty}_p(({\RE^d})^k)$ denotes the set of all infinitely differentiable real 
functions $f$ on $({\RE^d})^k$ such that all its partial derivatives are of at most polynomial growth. 
For $F$ as in the above definition of $Y$ the gradient operator $D$ is given by 
\begin{align*}
	D_sF(\gamma)=\sum_{i=1}^{k} (s_{i}\wedge s) (\nabla_{s_{i}}f)(\gamma), \quad s\in [0,1],\ 
	\gamma\in\Omega, 
\end{align*}
where $(\nabla_{s_{i}}f)(\gamma)=(\nabla_{s_{i}}f)(\gamma(s_1),\dots,\gamma(s_k))$ denotes the 
gradient of the function $f$ relative to the $i$th variable while holding the other variables 
fixed. Let $\CM$ denote the Cameron-Martin space, i.e. the space of all absolutely continuous 
$\RE^d$-valued functions $f$ on $[0,1]$, with $f(0)=0$, equipped with inner product
\begin{equation*}
\langle \varphi,\psi\rangle_\CM:=\int_{[0,1]} \langle  \varphi'(x) , \psi'(x) \rangle_{\RE^d}  
\,dx.
\end{equation*} 
Moreover let $S_i$, $i\in\N$, be the ONB 
in $\CM$ consisting of the Schauder functions.
Denote by $\mathcal{B}$ the $\sigma$-algebra of the Borel sets on $\Omega$. Let $\nu$ denote 
the Wiener measure on $(\Omega,\mathcal{B})$. Let $0<\lambda_1\le\lambda_2\le\ldots\, $ be a 
sequence of real numbers satisfying 
\begin{align}
	\label{eq:mc:lambda}
	&\sum_{m=0}^\infty  2^{-m}\lambda_{d2^m} < \infty. 
\end{align}
This paper is concerned with an Ornstein-Uhlenbeck type Dirichlet form $(\EF,D(\EF))$ obtained 
by the closure of the positive symmetric bilinear form 
\begin{align}
	\label{eq:mc:ef}
	\EF(F,G)=\int\sum_{i=1}^\infty \lambda_i\langle DF, S_i\rangle_\CM\, \langle DG, 
	S_i\rangle_\CM\, d\nu, 
	\quad F,G\in Y, 
\end{align} 
on $L^2(\nu)$. The Dirichlet form $(\EF,D(\EF))$ is a particular case of the set of Dirichlet 
forms associated with a certain class of infinite dimensional processes with unbounded 
diffusion introduced and studied in \cite{Lobus2004}, \cite{WangWu2008}, \cite{ChenWu2014}, 
and \cite{KarlssonLobus2014}. In particular, it has been shown in \cite{KarlssonLobus2014}, 
Proposition 4.2, that condition \eqref{eq:mc:lambda} is necessary and sufficient for closability 
of $(\EF,Y)$ on $L^2(\nu)$. Furthermore, its closure $(\EF,D(\EF))$ on $L^2(\nu)$ is 
quasi-regular. 

\subsection{Main results}

Let $G_i$, $i\in\N$, be a sequence of independent one-dimensional Ornstein-Uhlenbeck processes,  
i.e. we have $dG_i(t)=- G_i(t)dt+\sqrt{2}\,dW_i(t)$, $t\ge 0$, for a sequence of independent 
one-dimensional Wiener processes $W_i$, $i\in\N$. This choice of $G_i$ corresponds to the fact that $G_i$ is related to the one-dimensional analogue of $\EF$, namely $\EF^1(f,g)=\int f'g'\, \varphi dx$ where $\varphi$ denotes the standard normal density. In the present paper we consider the Dirichlet 
form corresponding to \eqref{eq:mc:ef} and show in Lemma \ref{lem:rc:gi:2} and Proposition 
\ref{rc:prop:4} that the associated stochastic process has the representation 
\begin{align}
	\label{intro:eq:2}
	X_t:=\sum_{i=1}^\infty G_i(\lambda_it)\cdot S_i\, ,\quad t\ge 0, 
\end{align}
where the right-hand side converges in the norm of $C_0([0,1];\R^d)$ almost surely. 
\medskip 

The main result of Section \ref{rc:section:1} is that, provided that 
\begin{align} 
	\label{eq:intro:1}
	\sum_{m=0}^\infty 2^{-\frac{m}{2}}\cdot\lambda_{d2^m}\cdot\left(\max_{d2^m+1< i\le d2^{m+1}} 
	|G_i(0)|+m^\hf\right) <\infty ,
\end{align} 
the sequence of processes $X_t^{(n)}:=\sum_{i=1}^n G_i(\lambda_it)\cdot S_i$, $t\ge 0$, converges 
in distribution to $X_t$, $t\ge 0$, in the space $C_{C_0([0,1]; \R^d)}([0,\infty))$ of all 
continuous $C_0([0,1]; \R^d)$-valued trajectories. 
\medskip 

In Section \ref{qv:section:1} we determine the scalar quadratic variation of $X$. For this let $L^1([0,1];\R^d)$ denote the set of all integrable $\R^d$-valued functions $\gamma$ on 
$[0,1]$ and let $L^1([0,1];\R^d)$ be endowed with the norm  $\|\gamma\|_1:=\sum_{i=1}^d\int_0^1|\gamma_i|\, dx$ where $\gamma=(\gamma_1,\ldots ,\gamma_d)$. We consider 
$X$ as a process with either $C_0([0,1]; \R^d)$-valued or $L^1([0,1];\R^d)$-valued trajectories.   
Correspondingly, let $\|\, \cdot\, \|$ denote the norm in either $C_0([0,1]; \R^d)$ or 
$L^1([0,1];\R^d)$. Let $\xi_i$, $i\in \N$, be a sequence of independent identically distributed 
standard normal random variables and define 
\begin{align*}
	\theta:=2E\left[\left\|\sum_{i=1}^\infty\lambda_i^\hf \xi_iS_i\right\|^2\right]
\end{align*} 
where we note that $\theta$ depends on the choice of the norm. Suppose now that 
\begin{align}
	\label{eq:intro:2}
	\sum_{m=0}^\infty 2^{-\frac{m}{2}}\cdot\lambda_{d2^{m+1}}\cdot\left(\max_{d2^m+1< i 
	\le d2^{m+1}}|G_i(0)|+m^\hf\right)<\infty 
\end{align} 
and note the similarity to \eqref{eq:intro:1}. Let $\tau^n=\{0=t_0^n,t_1^n,\ldots,t_{k(n)}^n=T\}$, $n\in \N$, be an arbitrary sequence of partitions on $[0,T]$, $T>0$ such that 
$\lim_\nti |\tau^n|=0$. We show that in the ucp sense (uniform convergence in probability)
\begin{align*}
	[X]_t:=\lim_{\nti}\sum_{j:t_j^n\le t}\left\| X_{t_j^n}-X_{t_{j-1}^n}\right\|^2 
	=\theta t, 
\end{align*} 
cf. Proposition \ref{qv:cla:prop:1}. Furthermore in Proposition \ref{qv:reg:prop:1} we verify 
that for the scalar quadratic variation relative to the stochastic calculus via regularization, 
cf. \cite{RussoVallois2007}, \cite{Girolami2014}, and \cite{GirolamiRusso2014}, it holds that 
\begin{align*}
	\frac{1}{\delta} \int_0^t \left\|X_{s+\delta}-X_s\right\|^2 \, ds \stack {\delta\to 0}{\lra} 
	t\theta \quad \text{ucp on } t\in [0,\infty). 
\end{align*}
\medskip

The aim of Section \ref{tv:section:1} is to determine the tensor quadratic variation of the 
process $X$ in the Banach space $\lptp$ relative to the stochastic calculus via regularization. 
That is, we examine
\begin{align*}
	[X]^\otimes_t:=\lim_{\delta\to 0} \int_0^t \frac{(X_{u+\delta}-X_u)\otimes (X_{u+\delta} 
	-X_u)}{\delta} \,du
\end{align*}
in the ucp sense with respect to the norm in $\lptp$. For this, let $\xi_i$, $i\in \N$, be a 
sequence of independent standard normal random variables and define
\begin{align*}
	\Theta:=2E\left[\left(\sum_{i=1}^\infty \lambda_i^\hf \xi_i S_i \right)\otimes \left( 
	\sum_{i'=1}^\infty \lambda_{i'}^\hf\xi_{i'} S_{i'} \right)\right].
\end{align*} 
Provided that we have \eqref{eq:intro:2} we prove  
\begin{align*}
	[X]^\otimes_t= t\Theta ,\quad t\in [0,\infty). 
\end{align*}

Section \ref{it:section:1} is devoted to the It\^o formula corresponding to \eqref{intro:eq:2}. 
Using recent results presented in \cite{Girolami2014} and \cite{GirolamiRusso2014}, our efforts 
in Sections \ref{qv:section:1} and \ref{tv:section:1} lead to
\begin{align*}
	F(t,X_t)&=F(0,X_0)+\int_0^t \frac{\partial}{\partial s} F(s,X_s)\, ds+\int_0^t 
	{}_{B^\ast}\langle DF(s,X_s),dX_s\rangle_B \\
	&\quad+\hf \int_0^t {\vphantom{\big(}}_{(B\ptp B)^\ast}\left\langle D^2F(s,X_s), 
	\Theta\right\rangle_{(B\ptp B)^{\ast\ast}}\, ds
\end{align*}
where $B=L^1([0,1];\R^d)$. We also verify that, for a certain class of cylindrical functions 
$F$, the expression on the right hand side takes the well known form of the finite dimensional 
It\^o formula.
\medskip

We would like to emphasize that infinite dimensional Ornstein-Uhlenbeck processes of the form 
\eqref{intro:eq:2} allow a decomposition $X_t=Y_t+Z_t+A_t$, $t\ge 0$, where 
\begin{align*}
	Y_t=\sum_{i=1}^\infty \int_{u=0}^te^{-\lambda_iu}\, dW_i\left(e^{2\lambda_iu}-1\right) 
	\cdot S_i, 
\end{align*} 
$Z_t=-\sum_{i=1}^\infty \int_{u=0}^t\lambda_ie^{-\lambda_iu}W_i\left(e^{2\lambda_iu}-1\right) 
\, du\cdot S_i$, and	$A_t=\sum_{i=1}^\infty e^{-\lambda_it}G_i(0)\cdot S_i$. This 
decomposition is reasonable since the process $Y$ has independent increments. Independent 
increments are particularly useful for calculating a quadratic variation. Furthermore under 
mild conditions, the parts $Z$ and $A$ have a zero quadratic variation. In addition, we observe 
that the components $G_i(\lambda_it)\cdot S_i$ of 
$X_t=\sum_{i=1}^\infty G_i(\lambda_it)\cdot S_i$ are independent. In order to carry out the 
technical calculations of Sections \ref{rc:section:1}--\ref{tv:section:1} these independences 
are crucial. 

Several times we use results from extreme value theory which in part we derive in the appendix.

\subsection{Some basic definitions} 

Let $\Omega\equiv C_0([0,1];\R^d):=\{f\in C([0,1];\R^d):f(0)=0\}$ be the space of trajectories, endowed with the norm 
$\|f\|:=\sup\{|f_i(x)|:x\in [0,1],\ i\in\{1,\ldots ,d\}\}$ where $f=(f_1,\ldots ,f_d)$ and 
$f_1,\ldots ,f_d\in C_0([0,1];\R)$. Furthermore, let $L^1([0,1];\R^d)$ be endowed with the 
norm $\|f\|_1:=\sum_{i=1}^d\int_0^1|f_i|\, dx$ where $f=(f_1,\ldots ,f_d)$ and $f_1,\ldots , 
f_d\in L^1([0,1];\R)$. This choice of the norm in $L^1([0,1];\R^d)$ guarantees compatability with the tensor product in Section \ref{tv:section:1}.

Let $(e_j)_{j=1,\dots,d}$ denote the standard basis in $\RE^d$ and introduce the system 
of the Haar functions on $[0,1]$ by 
\begin{align}
\begin{split}
\label{eq_def_haar1}
&H_1(t)=1,\quad t\in [0,1], \\
  &H_{2^m+k}(t) = \left\{
  \begin{array}{l l l}
    2^{\frac{m}{2}} & \quad \text{if $t \in \left[\frac{2k-2}{2^{m+1}},\frac{2k-1}{2^{m+1}}\right)$}\\ 
	\vspace{-.3cm} & & \\ 
    -2^{\frac{m}{2}} & \quad \text{if $t \in \left[\frac{2k-1}{2^{m+1}}, \frac{2k}{2^{m+1}}\right)$} 
	\qquad k=1,\dots,2^m, \ m=0,1,\dots\, .\\ \vspace{-.3cm} & & \\ 
    0 & \quad \text{otherwise}
  \end{array} \right.
  \end{split}
\end{align}
Also define 
\begin{equation*}
\label{eq_def_gn}
g_{d(r-1)+j}:=H_r \cdot e_j, \quad r\in \N, \ j\in \{1,\dots,d\},
\end{equation*} 
and 
\begin{align}	
	\label{intro:eq:1}
	S_i(s):=\int_0^s g_i(u)\, du,\quad\mbox{\rm as well as}\quad\langle S_i,\gamma \rangle:= 
	\int_0^1 g_i(u)\, d\gamma_u.
\end{align} 

Introduce
\begin{align*}
\tilde{Y}:=&\left\{F(\gamma)=f\left(\langle S_1,\gamma \rangle,\dots,\langle S_k,\gamma \rangle 
\right): f \in C^{\infty}_p({\RE}^k),\, k\in \N \right\} 
\end{align*} 
as well as 
\begin{align*}
\tilde{Y}_n:=\left\{F(\gamma)=f\left(\langle S_1,\gamma \rangle,\dots,\langle S_k, 
\gamma \rangle\right): f \in C^{\infty}_p({\RE}^k),\, k\le n \right\}, \quad n\in \N.
\end{align*} 

\begin{lemma}
	We have $Y=\tilde{Y}$.
\end{lemma}
\begin{proof}
For the sake of clarity we concentrate on the case $d=1$. We note that for any dyadic point 
$s\in [0,1]$ we have $\gamma(s)=\int_0^1 \I_{[0,s]}(u)\, d\gamma_u$. We also note that 
$\langle S_i,\gamma \rangle$ can be written as
	\begin{align*}
		&\langle S_i,\gamma \rangle=\int_0^1 H_i(u) \, d\gamma_u\\
		&\qquad=2c\int_0^1 \I_{[0,s_2]}(u)\, d\gamma_u-c\int\I_{[0,s_1]}(u) 
		\, d\gamma_u-c\int\I_{[0,s_3]}(u)\, d\gamma_u
	\end{align*}
where for $i=2^m+k$, $m=0,1,\ldots\, $, $k=1,\ldots\, ,2^m$, we have $c=2^{\frac{m}{2}}$, 
$s_1=(k-1)/2^m$, $s_2=(2k-1)/2^{m+1}$, and $s_3=k/2^m$. Thus, any element of $\tilde{Y}$ 
belongs to $Y$. As for the converse direction, we note that for any dyadic $s\in [0,1]$ it 
holds that
	\begin{align*}
		 \I_{[0,s]}(u)=\sum_{i=1}^{N_s} \alpha_i H_i(u) 
	\end{align*} 
for some $N_s\in \N$ and $\alpha_1,\ldots ,\alpha_{N_s}\in\R$. The statement follows.
\end{proof}




\section{Finite dimensional approximation of the infinite dimensional process} 
\label{rc:section:1}

In this section we turn our attention to finite dimensional Dirichlet forms $(\EF_n,D(\EF_n))$, 
$n\in \N$, on suitable finite dimensional subspaces of $L^2(\nu)$ given by the closure of 
\begin{align}
	\label{eq:rc:form}
		\EF_n(F,G)=\int\sum_{i=1}^n\lambda_i\langle DF, S_i\rangle_\CM\, \langle DG, 
		S_i\rangle_\CM \, d\nu,\quad 
		F,G\in \tilde{Y}_n, 
\end{align}
where $0<\lambda_1\le\lambda_2\le \ldots$ is a sequence of constants. We assume that 
for a sequence of independent one-dimensional standard Wiener processes $W_i$, $i\in \N$, we have 
	\begin{align}
		\label{eq:rc:ou}
		dG_i(t)=- G_i(t)dt+\sqrt{2}\,dW_i(t),\quad t\ge 0.
	\end{align}
In particular, with a non-random initial value $G_i(0)$ the variable $G_i(\lambda_it)$ is 
$N(G_i(0)e^{-\lambda_it},1-e^{-2\lambda_it})$-distributed, $t>0$. In other words the $G_i$ are independent Ornstein-Uhlenbeck processes on $\R$ .
Therefore we have
\begin{align*}
	dG_i(\lambda_it)=-\lambda_iG_i(\lambda_i t)\, dt+\sqrt{2\lambda_i}\, dV_i(t),\quad i\in \N,
\end{align*}
where $V_i=1/\sqrt{\lambda_i}\ W_i(\lambda_i t)$ is a Wiener process. Consequently $G_i(\lambda_i \cdot)$ is associated with the $L^2(\R)$-generator $f\equiv f(x)\mapsto \lambda_i f''(x)-\lambda_i x f'(x)$ and hence with the bilinear form $(f,g)\mapsto \lambda_i\int f'(x)g'(x)\varphi(x)dx$, $f,g\in C_p^\infty(\R)$. Since $S_i$, $i\in \N$, is an orthonormal basis in the Cameron-Martin space and the $G_i$s are independent, the stochastic process
	\begin{align}
		\label{rc:intro:eq:1}
		X_t^{(n)}=\sum_{i=1}^n  G_i(\lambda_it)\cdot S_i,\quad t\ge 0,\ n\in\N,  
	\end{align}
is associated with the bilinear form $(\EF_n,D(\EF_n))$ given by \eqref{eq:rc:form}.
On the other hand, we recall from \cite{KarlssonLobus2014}, Proposition 5.2, that 
$(\EF,D(\EF))$ given by the closure of 
	\begin{align}
		\EF(F,G)=\int\sum_{i=1}^\infty \lambda_i\langle DF, S_i\rangle_\CM\, \langle DG, 
		S_i\rangle_\CM \, d\nu, 
		\quad F,G\in \tilde{Y}, 
	\end{align} 
is a quasi-regular Dirichlet form on $L^2(\nu)$ provided that $\sum_{m=0}^\infty 2^{-m}\cdot 
\lambda_{d2^m}<\infty$. The objective of this section is to identify the stochastic process which 
is associated with $(\EF,D(\EF))$ and to approximate it by the sequence $X^{(n)}$. Lemma 
\ref{lem:rc:gi:2} - Proposition \ref{rc:prop:4} are devoted to the well-definiteness and 
association with $(\EF,D(\EF))$ of the limit process.   

\begin{lem}
	\label{lem:rc:gi:2}
Assume that the processes $G_i(\lambda_it)$, $t\ge 0$, $i\in \N$, have been started with  
non-random initial values $G_i(0)\in \R$, $i\in \N$, such that $\sum_{i=1}^\infty |G_i(0)|S_i(s)$ 
converges in $C_0([0,1];\R^d)$. Let $X_0:=\sum_{i=1}^\infty G_i(0)S_i(s)$. Then almost surely the 
sum $\sum_{i=1}^\infty |G_i(\lambda_it)|\cdot S_i$, and hence 
	\begin{align}
		\label{rc:lem:1:eq:1}
		X_t:=\sum_{i=1}^\infty G_i(\lambda_it)\cdot S_i,
	\end{align}
converges in $C_0([0,1];\R^d)$ uniformly for all $t\ge 0$.
\end{lem}
\begin{proof}
We know that $G_i(\lambda_it)$ is $N(G_i(0)e^{-\lambda_it},1-e^{-2\lambda_it})$-distributed and 
that for any standard normal random variable $\xi$ we have  $P(\xi>x)\le 
\frac{1}{\sqrt{2\pi}x}e^{-\frac{x^2}{2}}$, $x>0$, see e.g. \cite{MortersPeres}, Appendix II, 
Lemma 3.1. Thus it holds for all $i\in \N$, all $t\ge 0$, and all $y>|G_i(0)|e^{-\lambda_it}$ 
that
	\begin{align*}
		P(|G_i(\lambda_it)|>y)\le \frac{\sqrt{2}}{\sqrt{\pi}\left(y-|G_i(0)|e^{-\lambda_it} 
		\right)}\cdot\exp\left\{-\hf\left(y-|G_i(0)|e^{-
		\lambda_it}\right)^2\right\}.
	\end{align*} 
Next we track down the numbering of the Schauder functions $S_i$, $i\in \N$. Let $i=d(r-1)+j$, 
$j\in\{1,\ldots,d\}$, and for $r\ge 2$, let $r=2^m+k$ where $m\in \{0,1,\ldots\, \}$, $k\in 
\{1,\ldots,2^m\}$. Now set $n(i):=m$ if $i\ge 2d+1$. We obtain 
	\begin{align*}
		&\sum_{i=2d+1}^\infty P\left(|G_i(\lambda_it)|>|G_i(0)|e^{-\lambda_it}+\sqrt{2n(i)}\right)
		\le \frac{1}{\sqrt{\pi}}\sum_{i=2d+1}^\infty\frac{1}{\sqrt{n(i)}}e^{-n(i)}\\
		&\qquad=\frac{1}{\sqrt{\pi}}\sum_{m=1}^\infty d2^m \frac{1}{\sqrt{m}}e^{-m}<\infty.
	\end{align*}
The Borel-Cantelli lemma states now that $\nu$-a.s. there is an $i_0$ such that 
	\begin{align*}
		|G_i(\lambda_it)|\le |G_i(0)|+\sqrt{2n(i)},\quad i>i_0.
	\end{align*}
Thus for $m$ such that $d2^m+1>i_0$ we have
	\begin{align}
		\notag
		&\left|\sum_{i=d2^m+1}^\infty G(\lambda_it)S_i(s)\right| 
		\le\sum_{i=d2^m+1}^\infty |G(\lambda_it)|S_i(s)
		\le \sum_{i=d2^m+1}^\infty|G_i(0)|S_i(s)+\sum_{i=d2^m+1}^\infty \sqrt{2n(i)}S_i(s)\\
		\label{rc:lem:2:eq:1}
		&\qquad\le \sum_{i=d2^m+1}^\infty|G_i(0)|S_i(s)+\sum_{m'=m}^\infty \sqrt{2m'} 
		2^{-\frac{m'+2}{2}}, \quad s\in [0,1].
	\end{align}
We have used the fact that for any $s\in [0,1]$ there is just one $d2^{m'}<i\le d2^{m'+1}$ 
such that $S_i(s)\neq 0$. The last chain of inequalities implies the lemma.
\end{proof}

We are now interested in the semigroup associated with the process \eqref{rc:lem:1:eq:1}. 
%
For a point ${\bf y}\equiv\sum_{i=1}^\infty y_iS_i\in C_0([0,1];\R^d)$ with $y_i\in \R$, $i\in 
\N$, introduce
\begin{eqnarray*}
T_t{\bf y}:=\sum_{i=1}^\infty y_ie^{-\lambda_it}S_i
\end{eqnarray*}
and
\begin{eqnarray*}
U_t{\bf y}:=\sum_{i=1}^\infty y_i\, \left(1-e^{-2\lambda_it}\right)^{\frac12}\, S_i\, ,\quad 
t>0.
\end{eqnarray*}
Define $\mu_t:=\nu\circ U_t^{-1}$, $t>0$.
Let $C_b(C_0([0,1];\R^d))$ denote the space of all bounded continuous functions on $C_0([0,1];\R^d)$.

\begin{proposition}\label{rc:prop:3}
({\it Mehler semigroup representation}) Let $X_0:=\sum_{i=1}^\infty G_i(0) S_i$ be non-random 
such that $\sum_{i=1}^\infty |G_i(0)| S_i$ converges in $C_0([0,1];\R^d)$, cf. Lemma 
\ref{lem:rc:gi:2}. Furthermore let $X_t$ be given by \eqref{rc:lem:1:eq:1} and consider an arbitrary $F\in C_b(C_0([0,1];\R^d)) 
\cup \tilde{Y}$. Under these conditions we have
\begin{eqnarray*}
E_{X_0}[F(X_t)]=\int F\left(T_tX_0+{\bf y}\right)\, \mu_t(d{\bf y})\, ,\quad t>0.
\end{eqnarray*}
\end{proposition}
\begin{proof}
	It is sufficient to prove the claim for all cylindrical 
	functions $F\in\tilde{Y}$ since $C_b(C_0([0,1];\R^d)) 
\cap \tilde{Y}$ is dense in $C_b(C_0([0,1];\R^d))$ with respect to the $L^1(\nu)$-norm by the following reasoning. 
According to the approximation property as presented in \cite{Bauer} Theorem I.5.7 the simple cylindrical functions 
are dense in the set of all simple functions on $C_0([0,1];\R^d)$ with respect to the $L^1(\nu)$-norm. Noting that 
the simple functions on $C_0([0,1];\R^d)$ are dense in $L^1(\nu)$ and simple cylindrical function are dense in $\tilde{Y}$ with respect to the $L^1(\nu)$-norm
we may conclude that $C_b(C_0([0,1];\R^d)) \cap \tilde{Y}$ is dense in $C_b(C_0([0,1];\R^d))$.
\medskip

For $i\in \N$ and $t>0$ let
	\begin{eqnarray*}
		\psi^{(i)}_{t,G_i(0)}(x)\, ,\ x\in \R,\quad\mbox{\rm be the density of}\quad
		N\left(G_i(0)e^{-\lambda_it},1-e^{-2\lambda_it}\right)
	\end{eqnarray*}
	and note that $\psi^{(i)}_{t,G_i(0)}$ is the density of the process $G_i(\lambda_it)$ at 
	time $t>0$ when started in $G_i(0)$. We have 
	\begin{eqnarray*}
		E_{X_0}[F(X_t)]&&\hspace{-.5cm}=E_{X_0}\left[F\left(\sum_{i=1}^\infty 
		G_i(\lambda_it)S_i\right)\right]\vphantom{\int} \\
		&&\hspace{-.5cm}=E_{X_0}\left[f\left(G_1(\lambda_1t),\ldots ,G_k(\lambda_kt)\right) 
		\right]\vphantom{\int} \\
		&&\hspace{-.5cm}=\int\ldots\int f\left(y_1,\ldots ,y_k\right)\psi^{(1)}_{t,G_1(0)} 
		(y_1)\ldots \psi^{(k)}_{t,G_k(0)}(y_k)\, dy_1\ldots\, dy_k\vphantom{\int} \\
		&&\hspace{-.5cm}=\int\ldots\int f\left(y_1+G_1(0)e^{-\lambda_1t},\ldots ,y_k+
		G_k(0)e^{-\lambda_kt}\right)\psi^{(1)}_{t,0}(y_1)\ldots \psi^{(k)}_{t,0}(y_k)\,
		dy_1\ldots\, dy_k\vphantom{\int} \\
		&&\hspace{-.5cm}=\int F\left(T_tX_0+{\bf y}\right)\, \mu_t(d{\bf y}).
	\end{eqnarray*}
\end{proof}
	
\begin{prop} 
\label{rc:prop:4}
Suppose $\sum_{m=0}^\infty 2^{-m}\cdot \lambda_{d2^m}<\infty$. Then the process 
\eqref{rc:lem:1:eq:1} is associated 
with the Dirichlet form $(\EF,D(\EF))$ in the sense of 
\begin{align*}
	\lim_{t\to 0}\int\frac{F -E_{\cdot}[F(X_t)]}{t}\cdot H\, d\nu=\EF(F,H),\quad F,H\in\tilde{Y}. 
\end{align*}
\end{prop}
\begin{proof} 
Let $\varphi$ denote the density of $N(0,1)$. Using the notation of the proof of 
Proposition \ref{rc:prop:3} and $H({\gamma})=h\left(\langle S_1, \gamma\rangle,\ldots ,\langle 
S_k,\gamma\rangle\right)$ all we have to verify turns out to be a well-studied calculation in finite 
dimension, 
	\begin{align*} 
		&\lim_{t\to 0}\int\frac{F -E_{\cdot}[F(X_t)]}{t}\cdot H\, d\nu \\ 
		&\quad=\lim_{t\to 0}\frac{1}{t}\int \left(\int f(x_1,\ldots,x_k)-f(y_1+x_1e^{-\lambda_1t}, 
		\ldots,y_k+x_ke^{-\lambda_kt}) \times\right. \\
		&\left.\vphantom{\int}\qquad\times\psi_{t,0}^{(1)}(y_1)\ldots \psi_{t,0}^{(k)}(y_k)\, 
		dy_1,\ldots dy_k \right)
		\cdot h(x_1,\ldots,x_k)\, \varphi(x_1)\ldots\varphi(x_k)\, dx_1,
		\ldots dx_k \vphantom{\int}\\
		&\quad=\int \sum_{i=1}^k \lambda_i\left\langle DF,S_i\right\rangle_\CM  \left\langle DH,S_i\right\rangle_\CM 
		 d\nu=\EF(F,H).
	\end{align*}
\end{proof}

We continue with the first step toward the approximation of the infinite dimensional process  
\eqref{rc:lem:1:eq:1} by finite dimensional processes of the form \eqref{rc:intro:eq:1}. 

\begin{prop}
\label{rc:prop:1}
Let $(\EF_n,D(\EF_n))$, $n\in \N$, be given by \eqref{eq:rc:form} and let $X^{(n)}$ denote 
the associated process, see \eqref{rc:intro:eq:1}. Suppose 
\begin{align}
\label{eq:rc:lambda:1}
\sum_{m=0}^\infty 2^{-\frac{m}{2}}\cdot \lambda_{d2^m}\cdot m^{\hf}<\infty
\end{align} 
and choose an initial value $X^{(n)}(0):= \sum_{i=1}^\infty G_i(0)S_i(s)$ such that 
\begin{align}
	\label{eq:rc:prop:as:1}
	\sum_{m=0}^\infty 2^{-\frac{m}{2}}\, \max_{d2^m+1< i\le d2^{m+1}}|G_i(0)|\cdot 
	\lambda_{d2^m}<\infty. 
\end{align} 
Let $P_n$, $n\in \N$, denote the distribution of the process $X^{(n)}$ on the space $C_{C_0([0,1];\R^d)} 
([0,\infty))$ of continuous $C_0([0,1];\R^d)$-valued trajectories on $t\in [0,\infty)$. Then 
the family $P_n$, $n\in \N$, is relatively compact in the space of all probability measures on $C_{C_0([0,1];\R^d)} 
([0,\infty))$, endowed with the Prokhorov metric.
\end{prop} 
\begin{proof}
Let $t\ge 0$ and $u\ge 0$. Conditioning on $\F_t:=\sigma(X_s:s\le t)$ and taking into account the Markov property, 
the random variable 
	\begin{align*}
	Y_i:=G_i(\lambda_i(t+u))-G_i(\lambda_it)
	\end{align*}
 follows a $ N(\mu_i,\sigma_i^2)$-distribution where
	\begin{align}
		\label{eq:rc:mu_sigma}
		\mu_i=G_i(\lambda_it)\cdot(e^{-\lambda_i u}-1) \text{ and } \sigma_i^2=1-e^{-2\lambda_iu}.
	\end{align}
We aim to prove the claim by a conclusion similar to, for example, the final one in 
\cite{LyonsZhang1994}, Theorem 3.1.1. We have
	\begin{align}
		\label{eq:rc:mi:0}
		&E\left[\|X_{t+u}^{(n)}-X_t^{(n)}\|^4\right]
		=E\left[\left( \sup_{s\in[0,1]} \left| \sum_{i=1}^n Y_i\cdot S_i(s)\right| \right)^4 
		\right].
	\end{align}
where the norm $\|\, \cdot\, \|$ is taken in $C_0([0,1];\R^d)$. 
It follows that
	\begin{align}
		\notag
		\left( \sup_{s\in[0,1]} \left| \sum_{i=1}^n Y_i\cdot S_i(s)\right| \right)^4 	
		&\le \left(|Y_1|+ \sup_{s\in[0,1]} \left| \sum_{m=0}^\infty \max_{d2^m< i\le d2^{m+1}} 
		Y_i\cdot S_i(s)\right| \right)^4\\
		\label{eq:rc:mi:0.5}
		&\le  8 Y_1^4+8\left(\sum_{m=0}^\infty 2^{-\frac{m+2}{2}} \max_{d2^m< i\le d2^{m+1}} 
		|Y_i| \right)^4.
	\end{align}
We have
	\begin{align}
		\label{eq:rc:mi:1}
		&\nquad \nquad E\left[\left(  \sum_{m=0}^{\infty} 2^{-\frac{m+2}{2}} \max_{d2^m< i\le d2^{m+1}} |Y_i| 
		\right)^4\right] 
		=\!\!\!\sum_{m_1,\ldots ,m_4=0}^{\infty}\ \prod_{j=1}^4\left( E\left[2^{-\frac{a_j(m_j+2)}{2}} 
		\cdot  \max_{d2^m_j< i\le d2^{m_j+1}} |Y_i|^{a_j} \right]\right)^{\frac{1}{a_j}}
	\end{align} 
where 
$a_j\equiv a_j(m_1,m_2,m_3,m_4)$ 
is the number of times $m_j$ appears among $m_1,m_2,m_3,m_4$. Now decompose $Y_i$ in the following manner. Let $\xi_i\sigma_i+\mu_i=Y_i$, 
$k\in \{1,\ldots,4\}$, where $\xi_i$ is a standard normal random variable and $\mu_i$ and 
$\sigma_i$ are given by \eqref{eq:rc:mu_sigma}. Taking the conditional expectation gives
	\begin{align*}
		&E\left[2^{-\frac{k(m+2)}{2}}\cdot \left(\max_{d2^m< i\le d2^{m+1}}|Y_i|\right)^k \right] 
		=E\left[2^{-\frac{k(m+2)}{2}} \cdot E\left[ \max_{d2^m< i\le d2^{m+1}}|\xi_i\sigma_i+ 
		\mu_i|^k \Big|\F_t\right]\right].
	\end{align*}
We recall that both, $\mu_i$, $i\in \N$, and $\xi_i$, $i\in \N$, are by construction sequences 
of independent random variables. Moreover, $\sigma_i>0$ is non-random, $\mu_i$ is measurable with 
respect to $\F_t$, and $\xi_i$ is independent of $\F_t$. It follows that 
	\begin{align*}
		&\nquad \nquad E\left[2^{-\frac{k(m+2)}{2}} \cdot  \max_{d2^m< i\le d2^{m+1}} 
		|Y_i|^k \right]\\
		&	\le  8\cdot 2^{-\frac{k(m+2)}{2}} \cdot \max_{d2^m< i\le d2^{m+1}} 
		\sigma_i^k\cdot E\left[\max_{d2^m< i\le d2^{m+1}} |\xi_i|^k \right]+8E\left[ 
		2^{-\frac{k(m+2)}{2}} \cdot \max_{2^m< i\le 2^{m+1}} |\mu_i|^k\right],
	\end{align*}
$k\in \{1,\ldots,4\}$. Now we apply $E\left[\max_{1\le i\le n} |\xi_i|^k \right]\le D 
(\ln n)^{\frac{k}{2}}$ for some $D>0$ independent of $n\ge 2$ which is a standard result from 
extreme value theory.  
We obtain
	\begin{align}
		\notag
		&E\left[2^{-\frac{k(m+2)}{2}} \cdot \max_{d2^m< i\le d2^{m+1}} |Y_i|^k \right]\\
		\label{eq:rc:1}
		&\qquad\le 8D_1\cdot 2^{-\frac{k(m+2)}{2}}\cdot \max_{d2^m< i\le d2^{m+1}}\sigma_i^k\cdot  
		m^{\frac{k}{2}}+8E\left[2^{-\frac{k(m+2)}{2}} \cdot \max_{2^m< i\le 2^{m+1}} |\mu_i|^k 
		\right], 
	\end{align}
$k\in \{1,\ldots,4\}$, for some $D_1>0$ independent of $m\ge 1$. Next we recall that
	\begin{align*}
		\mu_i \text{ is }N\left(G_i(0)e^{-\lambda_it}(1-e^{-\lambda_iu}),(1-e^{-2\lambda_it}) 
		(1-e^{-\lambda_iu})^2\right)\text{-distributed} 
	\end{align*}
and that $\sigma_i=\left(1-e^{-2\lambda_iu}\right)^\hf$. Therefore we may write 
	\begin{align*}
		\mu_i=G_i(0)e^{-\lambda_it}(1-e^{-\lambda_iu})+\left(1-e^{-2\lambda_it}\right)^\hf 
		(1-e^{-\lambda_iu})\cdot \eta_i,
	\end{align*}
where $\eta_i$, $i\in \N$, is a sequence of independent $N(0,1)$-distributed random variables. 
Applying this to \eqref{eq:rc:1} 
we obtain
	\begin{align}
	\notag
	&E\left[2^{-\frac{k(m+2)}{2}} \max_{d2^m< i\le d2^{m+1}} |Y_i|^k\right]\\
	\notag
		&\quad\le 8D_1\cdot 2^{-\frac{k(m+2)}{2}} \cdot \max_{d2^m< i\le d2^{m+1}} \sigma_i^k 
		\cdot m^\frac{k}{2} \\
		\notag
		&\qquad+8E\left[2^{-\frac{k(m+2)}{2}} \cdot \max_{d2^m< i\le d2^{m+1}} \left|G_i(0) 
		e^{-\lambda_it}(1-e^{-\lambda_iu})+\left(1-e^{-2\lambda_it}\right)^\hf 
		(1-e^{-\lambda_iu})\cdot \eta_i\right|^k \right]\\
		\notag
		&\quad\le 8D_1\cdot 2^{-\frac{k(m+2)}{2}}\cdot \left(1-\exp\left\{-2\lambda_{d2^m}\cdot u 
		\right\}\vphantom{l^1}\right)^\frac{k}{2}\cdot m^\frac{k}{2} \\
		\notag
		&\qquad+8^2\cdot 2^{-\frac{k(m+2)}{2}}\cdot \left(1-\exp\left\{-\lambda_{d2^m}\cdot u 
		\right\}\vphantom{l^1}\right)^k\left( \max_{d2^m< i\le d2^{m+1}}|G_i(0)|^k + D_1 \cdot 
		m^\frac{k}{2}\right) \\
		\label{eq:rc:fin:1}
		&\quad\le D_2\cdot 2^{-\frac{k(m+2)}{2}} \cdot \lambda_{d2^m}^k\cdot u^\frac{k}{2} \cdot 
		\left(\max_{d2^m< i\le d2^{m+1}}|G_i(0)|^k + m^\frac{k}{2}\right), \quad u\in [0,1], 
	\end{align}
$k\in \{1,\ldots,4\}$, for some $D_2>0$ independent of $m\ge 1$ and $t\ge 0$. Taking into 
account conditions \eqref{eq:rc:lambda:1} and \eqref{eq:rc:prop:as:1} we apply \eqref{eq:rc:fin:1} 
to \eqref{eq:rc:mi:1} and thus also to \eqref{eq:rc:mi:0} and obtain
	\begin{align}
	\label{eq:rc:fin:2}
	\nquad E\left[\|X_{t+u}^{(n)}-X_t^{(n)}\|^4\right]\le 8E[Y_1^4]+8E\left[\left( \sum_{m=0}^{\infty} 
	2^{-\frac{m+2}{2}} \max_{d2^m< i\le d2^{m+1}} |Y_i| \right)^4\right] \le C\cdot u^2,\quad 
	u\in [0,1], 
	\end{align}
for some $C>0$ independent of $n\in \N$ and $t\ge 0$. It follows now from a straight forward 
generalization of \cite{Billingsley}, Theorem 8.3, and \cite{StroockVaradhan}, Corollary 2.1.4, 
to Banach space valued variables, that the family $P_n$, $n\in \N$, is relatively compact in the 
space $C_{C_0([0,1];\R^d)}([0,\infty))$ of continuous trajectories.
\end{proof}

We now show convergence in distribution of $X_t^{(n)}=\sum_{i=1}^n  G_i(\lambda_it)\cdot S_i$, 
$t\ge 0,$ to $X_t:=\sum_{i=1}^\infty G_i(\lambda_it)\cdot S_i$ as $n\to\infty$. 

\begin{prop}
	\label{rc:prop:2} 
Under the assumptions of Proposition \ref{rc:prop:1} the following holds.\\ 
(a) The process $X_t:=\sum_{i=1}^\infty G_i(\lambda_it)\cdot S_i$, $t\ge 0$, has almost surely 
a continuous modification. We will use this modification from now on. \\ 
(b) Let $P_n$ and $P$ denote the distributions associated with the processes $X^{(n)}$ and $X$ 
respectively. The finite dimensional distributions of $X^{(n)}$ converge weakly to those of $X$. 
That is 
	\begin{align*}
		&\int \varphi(x_1,\ldots,x_m)\, dP_n(X^{(n)}(t_1)\in dx_1,\ldots,X^{(n)}(t_m)\in dx_m)\\
		&\quad\lni \int \varphi(x_1,\ldots,x_m)\, dP(X(t_1)\in dx_1,\ldots,X(t_k)\in dx_m)
	\end{align*}
for all $m\in \N$ and all functions	$\varphi$ belonging to some algebra of convergence determining 
functions over $\left(C_0([0,1];\R^d)\right)^m$, cf. \cite{EthierKurtz} Section 4.4.
\end{prop}
\begin{proof} 
(a) It follows from \eqref{eq:rc:mi:0}, \eqref{eq:rc:mi:0.5}, and \eqref{eq:rc:fin:2} that 
	\begin{align*}
	E\left[\|X_{t+u}-X_t\|^4\right]\le C\cdot u^2,\quad u\in [0,1], 
	\end{align*}
with the same $C>0$ as in \eqref{eq:rc:fin:2}. In fact, in \eqref{eq:rc:mi:0.5} we turn to 
calculations which hold uniformly for all $n\in \N$ and also for estimating $E\left[\|X_{t+u}-X_t 
\|^4\right]$. By  Kolmogorov's criterion, cf. 
\cite{DaPratoZabczyk}, Theorem 3.3, we get continuity of 
the trajectories of $X$ as claimed. Alternatively, part(a) follows also from Proposition 
\ref{qv:cla:prop:1} below  which is proved independently of the results of the present section. 
\medskip 

\noindent 
(b) 
By \cite{EthierKurtz} Theorem 3.4.5 (b) it is sufficient to show that
 	\begin{align}
		\label{rc:prop:2:eq:1}
		\int F_1\big(X^{(n)}(t_1)\big)\cdot\ldots\cdot F_m\big(X^{(n)}(t_m)\big)\, dP_n \lni 
		\int F_1\big(X(t_1)\big)\cdot\ldots\cdot F_m\big(X(t_m)\big)\, dP
 	\end{align}
for all $F_i\in \tilde{Y}$ of type $F_i({\bf x})=f_{i1}(\langle S_1,{\bf x}\rangle)\cdot \ldots 
\cdot f_{ik}(\langle S_k,{\bf x}\rangle)$, $i=1,\ldots,m$. It is sufficient since the functions $\varphi({\bf x}_1 
,\ldots ,{\bf x}_m):=F_1({\bf x}_1)\cdot\ldots\cdot F_m({\bf x}_m)$ generate a convergence 
determining algebra over $\left(C_0([0,1];\R^d)\right)^m$. 
First we construct the Mehler representation relative to $X^{(n)}$. Recalling our approach and 
our notation in Proposition \ref{rc:prop:3} we define for ${\bf y}^{(n)}=\sum_{i=1}^n y_iS_i\in 
C_0([0,1];\R^d)$ with $y_i\in\R$, $i\in\N$, 
	\begin{align*}
		&T_t^{(n)}{\bf y}^{(n)}:=\sum_{i=1}^n y_ie^{-\lambda_it}S_i\quad \text{and}\quad 
		U_t^{(n)}{\bf y}^{(n)}:=\sum_{i=1}^n y_i(1-e^{-2\lambda_it})^\hf S_i,\quad t\ge 0. 
	\end{align*}
Furthermore, let 
	\begin{align*}
		&\mu_t^{(n)}:=\nu^{(n)}\circ (U_t^{(n)})^{-1} 				
	\end{align*}
where $\nu^{(n)}$ is the projection of the Wiener measure to the linear span of $\{S_1,\ldots, 
S_n\}$. Let $\mathcal{T}_t^{(n)}$ denote the semi-group associated with $(\EF_n,D(\EF_n))$ and 
let and $X_0^{(n)}$ denote the non-random initial value of $X^{(n)}$, $n\in \N$. Let us consider 
cylindrical functions $F\in\tilde{Y}$ and $F_n\in\tilde{Y}$, $n\in \N$, of type $F({\bf x})=f_1(\langle S_1,{\bf x}\rangle)\cdot  
\ldots\cdot f_k(\langle S_k,{\bf x}\rangle)$ and $F_n({\bf x})=f_1(\langle S_1,{\bf x}\rangle) 
\cdot\ldots\cdot f_{n\wedge k}(\langle S_{n\wedge k},{\bf x}\rangle)$, for ${\bf x}\in C_0([0,1]; 
\R^d)$. Following the proof of Proposition \ref{rc:prop:3}, the Mehler representation of 
$\mathcal{T}_t^{(n)}F_n$ is 
	\begin{align}
		\label{rc:prop:1:eq:1}
		&\left(\mathcal{T}_t^{(n)}F_n\right)\big(X^{(n)}_0)=\int F_n(T_t^{(n)}X^{(n)}_0 
		+{\bf y}^{(n)})\, \mu_t^{(n)}(d{\bf y}^{(n)}) 
	\end{align} 
where we mention that $(\mathcal{T}_t^{(n)}F_n)\big(X^{(n)}_0)$ is of the form $g_1(\langle 
S_1,{\bf x}\rangle)\cdot\ldots\cdot g_{n\wedge k}(\langle S_{n\wedge k},{\bf x} \rangle)$. 
Using the fact that $F(T_t^{(n)}X^{(n)}_0+{\bf y}^{(n)})=F(T_tX_0+{\bf y})$ for $n\ge k$ we obtain 
	\begin{align*}
		\int F(T_t^{(n)}X^{(n)}_0+{\bf y}^{(n)})\, \mu_t^{(n)}(d{\bf y}^{(n)})
		=\int F(T_tX_0+{\bf y})\, \mu_t(d{\bf y}),\quad n\ge k.
	\end{align*}
Denoting by $\mathcal{T}_t$, $t\ge 0$, the semigroup associated with $(\EF,D(\EF))$, by 
\eqref{rc:prop:1:eq:1} it follows that $\mathcal{T}_t^{(n)}F = \mathcal{T}_tF$, $t\ge 0$, for 
sufficiently large $n\in\N$. Let us fix $F_1,\ldots,F_m \in \tilde{Y}$, $m\in \N$. For 
sufficiently large $n\in\N$ it holds that 
	\begin{align*}
		&\int F_1\left(X^{(n)}(t_1)\right)\cdot\ldots\cdot F_m\left(X^{(n)}(t_m)\right)\, 
		dP_n\\
		&\quad=\int\left(\mathcal{T}_{t_1}^{(n)}\left(F_1\cdot \mathcal{T}^{(n)}_{t_2-t_1} 
		\left(F_2\cdot \ldots \mathcal{T}^{(n)}_{t_{m}-t_{m-1}}(F_m)\ldots \right)\right) 
		\right)\left(X_0^{(n)}\right)\, \nu ^{(n)}\left(dX_0^{(n)}\right)\\
		&\quad=\int\left(\mathcal{T}_{t_1}\left(F_1\cdot \mathcal{T}_{t_2-t_1}\left(F_2\cdot 
		\mathcal{T}_{t_{m}-t_{m-1}}(F_m)\ldots\right)\right)\right)\left(X_0\right)\, \nu 
		\left(dX_0\right)\\
		&\quad=\int F_1(X(t_1))\cdot\ldots\cdot F_m(X(t_m))\, dP,
	\end{align*}
	i.e. we have \eqref{rc:prop:2:eq:1}. The statement follows.
\end{proof}

\begin{corollary} 
	Let $X^{(n)}$, $X$, $P_n$, and $P$ denote the processes and associated distributions 
	introduced in \eqref{rc:intro:eq:1} and \eqref{rc:lem:1:eq:1}. Under the hypotheses of 
	Proposition \ref{rc:prop:1} we have $P_n\lni P$ weakly on $C_{C_0([0,1]; \R^d)}([0,\infty))$.
\end{corollary}
\begin{proof}
	Combining the relative compactness result of Proposition \ref{rc:prop:1} with the weak 
convergence of the finite dimensional distributions shown in Proposition \ref{rc:prop:2} we 
prove the statement, see also Theorem 3.7.8 (b) of \cite{EthierKurtz}.
\end{proof}







\section{Scalar quadratic variation}
\label{qv:section:1}

We are now interested in a criterion for the increase of the sequence 
$\lambda_i$, $i\in \N$, in order to guarantee the finiteness of the quadratic variation. 
Subsequently we will use the concept of uniform convergence on compact subsets of $[0,\infty)$ 
in probability (abbreviated ucp), see e.g. \cite{Protter} II.4 Definition p. 57. 

Let $W_i$, $i\in \N$, be a sequence of independent one-dimensional standard Wiener processes and 
let $G_i(0)\in \R$, $i\in \N$, such that $\sum_{i=1}^\infty |G_i(0)|S_i$ converges in $C_0([0,1]; 
\R^d)$. Then $G_i$, $i\in \N$, defined by 
\begin{align*}
	G_i(t)&:=e^{-t}G_i(0)+e^{-t}W_i\left(e^{2t}-1\right) \\ 
	&=e^{-t}G_i(0)+\int_{u=0}^te^{-u}\, dW_i\left(e^{2u}-1\right)+\int_{u=0}^tW_i\left(e^{2u}-1 
	\right)\, de^{-u},\quad t\ge 0,
\end{align*}
is a sequence of independent one-dimensional Ornstein-Uhlenbeck processes. Let the processes $X$, 
$Y$, and $Z$ on $t\in [0,\infty)$ be given by
\begin{align*} 
	X_t(s):=\sum_{i=1}^\infty G_i(\lambda_it)\cdot S_i(s), 
\end{align*}
\begin{align}
\label{eq:qv:ex:y}
	Y_t(s)=\sum_{i=1}^\infty \int_{u=0}^te^{-\lambda_iu}\, dW_i\left(e^{2\lambda_iu}-1\right) 
	\cdot S_i(s), 
\end{align} 
and 
\begin{align}
\label{eq:qv:ex:z}
	Z_t(s)=-\sum_{i=1}^\infty \int_{u=0}^t\lambda_ie^{-\lambda_iu}W_i\left(e^{2\lambda_iu}-1 
	\right)\, du\cdot S_i(s),\quad s\in [0,1]. 
\end{align} 
provided that the infinite sums \eqref{eq:qv:ex:y} and \eqref{eq:qv:ex:z} converge almost surely 
in $C_0([0,1]; \R^d)$. Recall also Lemma \ref{lem:rc:gi:2}. In addition denote by 
$A_t:=\sum_{i=1}^\infty e^{-\lambda_it}G_i(0)\cdot S_i$, $t\ge 0$, the non-random part of 
\begin{align*}
	X_t=Y_t+Z_t+A_t,\quad t\ge 0. 
\end{align*} 
We aim to analyze the quadratic variation of $X$. For this let $T>0$ be arbitrary and consider 
an arbitrary sequence of partitions of $[0,T]$,
\begin{align*}
&\tau^n=\{t_0^n,t_1^n,\ldots,t_{k(n)}^n: 0=t_0^n\le t_1^n\le \ldots \le t_{k(n)}^n=T\},\quad n\in \N,
\end{align*}
with $\lim_\nti |\tau^n|=0$. Here $|\cdot|$ denotes the mesh size, $|\tau|:=\max_j(t_{j}-t_{j-1})$. 
Let $\|\, \cdot\, \|$ denote the norm in either $C_0([0,1]; \R^d)$ or $L^1([0,1];\R^d)$. For 
$0\le  t \le T$ we introduce the ucp limit
\begin{align}
\label{eq:qv:lim:x}
		[X]_t:=\lim_{\nti}\sum_{j:t_j^n\le t}\left\| X_{t_j^n}-X_{t_{j-1}^n}
		\right\|^2,
	\end{align} 
	where we note that this limit depends on the choice of the norm. In the same way we define
	\begin{align}
		\label{eq:qv:lim:y}
		[Y]_t:=\lim_{\nti}\sum_{j:t_j^n\le t}\left\|Y_{t_j^n}-Y_{t_{j-1}^n}\right\|^2 
	\end{align} 	
and $[Z]_t:=\lim_{\nti}\sum_{j:t_j^n\le t}\left\|Z_{t_j^n}-Z_{t_{j-1}^n}\right\|^2$ provided that 
these limits exists ucp. For $t\in [0,T]$ introduce  
\begin{align*}
		[A]_t:=\lim_{\nti}\sum_{j:t_j^n\le t}\ \left\| A_{t_j^n}-A_{t_{j-1}^n}\right\|^2 
\end{align*} 
whenever this limit exists uniformly in $t\in [0,T]$. 
 
Let $\xi_i$, $i\in \N$, be a sequence of independent identically distributed standard normal 
random variables and define
\begin{align*}
	\theta:=2E\left[\left\|\sum_{i=1}^\infty\lambda_i^\hf \xi_iS_i\right\|^2\right]. 
\end{align*}

\begin{lemma}
\label{lem:qv:var:1}
Suppose that 
\begin{align}
\label{eq:qv:lambda:1}
\sum_{m=0}^\infty 2^{-\frac{m}{2}}\cdot \lambda_{d2^{m+1}}^{\hf} \cdot m^{\hf}<\infty.
\end{align}
(a) The sum \eqref{eq:qv:ex:y} converges almost surely in $C_0([0,1]; \R^d)$. Furthermore, 
the limit \eqref{eq:qv:lim:y} exists ucp on $t\in [0,\infty)$ for both norms, 
$C_0([0,1];\R^d)$ as well as $L^1([0,1];\R^d)$. It holds that $\theta<\infty$ 
and we have $[Y]_t=t\theta$, $t\in [0,\infty)$.\\
(b) For any $T>0$ we have for both norms, $C_0([0,1];\R^d)$ as well as $L^1([0,1];\R^d)$,
\begin{align*}
	\var\left(\sum_{j:t^n_j\le T}\left\|Y_{t_j^n}-Y_{t_{j-1}^n}\right\|^2 \right) 
	\le C \sum_{j:t^n_j\le T}(t^n_j-t^n_{j-1})^2
\end{align*}
for some $C>0$ independent of $n\in \N$ and $\tau^n$.
\end{lemma}
\begin{proof}
Let $T>0$. For the sake of clarity, below we will use the abbreviations $t_j^n\equiv t_j$, and 
\begin{align*}
	\Delta_{i,j}Y:=\int_{u=t_{j-1}}^{t_j}e^{-\lambda_iu}\, dW_i\left(e^{2\lambda_iu}-1\right). 
\end{align*} 
We observe that $\Delta_{i,j}Y$ is $N\left(0,2\lambda_i(t_j-t_{j-1})\right)$-distributed. Therefore
\begin{align*}
		E\left[\left\|\sum_{i=1}^\infty \left| 
		\Delta_{i,j}Y\right|S_i\right\|^4\right]		
		 \le d E\left[\left(\max_{0<i\le d}\left|\Delta_{i,j}Y\right| 
		+\sum_{m=0}^\infty 2^{-\frac{m+2}{2}}\max_{d2^m<i\le d2^{m+1}}\left|\Delta_{i,j}Y\right| 
		\right)^4\right].
\end{align*}
In the case of the $C$-norm the last line is obvious since taking the $\sup$-norm of 
$\sum_{i=1}^\infty \big|\Delta_{i,j}Y\big|S_i$ we just focus on one $s\in [0,1]$. We recall also 
that the $L^1$-norm of $S_i$ is bounded by the $C$-norm of $S_i$. Taking the fourth power and 
using the independence of the different Wiener processes it follows that we have to investigate 
terms of type
	\begin{align*}
	E\left[\max_{d2^m<i\le d2^{m+1}}\big|\Delta_{i,j}Y\big|^k\right],\quad k\in \{1,\ldots,4\}.
	\end{align*}
A known result from extreme value theory says that for a sequence $\xi_i$, $i\in \N$, of 
independent $N(0,1)$-distributed random variables we have $E\left[\max_{1\le i\le r} 
|\xi_i|^k\right]\le c_k (\ln r)^\frac{k}{2}$ for some $c_k>0$ independent of $r\ge 2$. For example see 
\cite{EmKlMi}, Example 3.3.29, and \cite{Resnick}, Proposition 2.1 (iii), where $a_n=c_n$ and $b_n=d_n$. 
Recalling that $\lambda_i$ is an increasing sequence it follows that
	\begin{align*}
	&E\left[\max_{d2^m<i\le d2^{m+1}}\big|\Delta_{i,j}Y\big|^k\right]\le d_k\ 
	\lambda_{d2^{m+1}}^{\frac{k}{2}} \cdot m^{\frac{k}{2}}\cdot\left(t_j-t_{j-1}\vphantom{l^1} 
	\right)^{\frac{k}{2}},
\end{align*}
$k\in \{1,\ldots,4\}$ for a suitable $d_k>0$ independent of $m$ and $\lambda_i$, $i\in \N$. 
Piecing everything together and keeping \eqref{eq:qv:lambda:1} in mind we obtain
\begin{align}
\notag
	&E\left[ \left(\max_{0<i\le d}\big|\Delta_{i,j}Y\big|+\sum_{m=0}^\infty 2^{-\frac{m+2}{2}}
	\max_{d2^m<i\le d2^{m+1}}\big|\Delta_{i,j}Y\big|\right)^4\right]\\
	\label{eq:rc:p2}
	&\quad\le C_1 \cdot  \left(\lambda_d^\hf+\sum_{m=0}^\infty 2^{-\frac{m}{2}}\cdot 
	\lambda_{d2^{m+1}}^{\frac{1}{2}} \cdot m^{\frac{1}{2}}\right)^4\cdot \left(t_j-t_{j-1}
	\vphantom{l^1}\right)^{2} 
\end{align}
for some $C_1>0$ independent of $\lambda_i$, $i\in \N$, $n\in \N$, and $\tau^n$. We get finiteness in \eqref{eq:rc:p2} for any partition $\tau^n$. Thus  $E\left[\left\|\sum_{i=1}^\infty \Delta_{i,j}Y S_i\right\|^2\right]$ is also finite. Applying Kolmogorov's inequality we obtain
	\begin{align}
		\notag
		&P\left(\sup_{t\le T}\left| \sum_{j:t_j\le t}\left(\left\|\sum_{i=1}^\infty 
		\Delta_{i,j}YS_i\right\|^2-E\left[\left\|\sum_{i=1}^\infty \Delta_{i,j}Y 
		S_i\right\|^2\right]\right)\right| \ge \eps\right) \\
		\notag
		&\quad \le\frac{1}{\eps^2}\sum_j \var\left(\left\|\sum_{i=1}^\infty \Delta_{i,j}Y 
		S_i\right\|^2\right)\le\frac{1}{\eps^2}\sum_j E\left[\left\|\sum_{i=1}^\infty \left| 
		\Delta_{i,j}Y\right|S_i\right\|^4\right]\\
		\label{eq:rc:p1}
		&\quad \le\frac{d}{\eps^2}\sum_j E\left[\left(\max_{0<i\le d}\left|\Delta_{i,j}Y\right| 
		+\sum_{m=0}^\infty 2^{-\frac{m+2}{2}}\max_{d2^m<i\le d2^{m+1}}\left|\Delta_{i,j}Y\right| 
		\right)^4\right].
	\end{align}
We note that the finiteness of the right-hand side implies 
that the sum \eqref{eq:qv:ex:y} converges almost surely in $C_0([0,1]; \R^d)$. Continuing from 
\eqref{eq:rc:p1} and using \eqref{eq:rc:p2} we get
 	\begin{align*}
		\notag
		&P\left(\sup_{t\le T} \left|\sum_{j:t_j\le t}\left\|\sum_{i=1}^\infty 
		\Delta_{i,j}YS_i\right\|^2-E\left[\left\|\sum_{i=1}^\infty \Delta_{i,j}YS_i 
		\right\|^2\right]\right|\ge \eps\right) \\
		\notag
		&\quad \le \frac{dC_1}{\eps^2}\cdot\left(\lambda_d^\hf+\sum_{m=0}^\infty 2^{-\frac{m}{2}} 
		\cdot \lambda_{d2^{m+1}}^{\frac{1}{2}} \cdot m^{\frac{1}{2}}\right)^4 \sum_{j:t_j\le t} 
		\left(t_j-t_{j-1}\vphantom{l^1}\right)^{2} \stack{|\tau|\to 0}{\lra} 0.
	\end{align*}
It follows that for any sequence of partitions $\tau^n$, $n\in \N$ we have convergence ucp. We 
also note that
\begin{align}
	\notag
	&\lim_{|\tau|\to 0} \sum_{j:t_j\le t} E\left[\left\|\sum_{i=1}^\infty 
	\Delta_{i,j}YS_i\right\|^2\right]\\
	\notag
	&\quad=\lim_{|\tau|\to 0}\sum_{j:t_j\le t} E\left[2\left(t_j-t_{j-1}\vphantom{l^1}\right) 
	\left\|\sum_{i=1}^\infty\frac{\lambda_i^\frac12\Delta_{i,j}Y}{\left(\var( 
	\Delta_{i,j}Y)\right)^\hf}S_i\right\|^2\right]\\ 
	\label{eq:rc:p3}
	&\quad= t\theta.
\end{align}
This shows (a). Part (b) follows from \eqref{eq:rc:p1}, \eqref{eq:rc:p2}, and 
\eqref{eq:qv:lambda:1}.
\end{proof}

\begin{lemma}
\label{lem:qv:var:2}
(a) Suppose that 
\begin{align}
\label{eq:rc:q1}
	\sum_{m=0}^\infty 2^{-\frac{m}{2}}\, \max_{d2^m+1< i\le d2^{m+1}}|G_i(0)|\cdot 
	\lambda_{d2^{m+1}}<\infty. 
\end{align} 
Then $\sum_{i=1}^\infty |G_i(0)|S_i$ converges in $C_0([0,1];\R^d)$, i.e. $A_t=\sum_{i=1}^\infty 
e^{-\lambda_it}G_i(0)\cdot S_i$, $t\ge 0$, is well-defined. Furthermore, the limit $[A]_t= 
\lim_{\nti}\sum_{j:t_j^n\le t}\, \| A_{t_j^n}-A_{t_{j-1}^n}\|^2$ exists uniformly in $t\in [0,T]$ 
for both norms, $C_0([0,1];\R^d)$ and $L^1([0,1];\R^d)$, and is constant zero on $[0,T]$. 
\medskip 

\noindent
(b) Suppose that 
\begin{align}
\label{eq:rc:q2}
\sum_{m=0}^\infty 2^{-\frac{m}{2}}\cdot \lambda_{d2^{m+1}} \cdot m^{\hf}<\infty. 
\end{align} 
Under this condition the sum \eqref{eq:qv:ex:z} converges almost surely in $C_0([0,1]; \R^d)$. Moreover, it holds that the 
limit $[Z]_t=\lim_{\nti}\sum_{j:t_j^n\le t}$ $\| Z_{t_j^n}-Z_{t_{j-1}^n}\|^2$ exists ucp on $t\in 
[0,\infty)$ for both norms, $C_0([0,1];\R^d)$ as well as $L^1([0,1];\R^d)$. Furthermore the limit is constant zero. 
\end{lemma}
\begin{proof}
Similar as in the proof of the previous lemma, below we will use the abbreviations $t_j^n\equiv 
t_j$, and $\Delta_{i,j}A:=G_i(0)e^{-\lambda_i t_j}-G_i(0)e^{-\lambda_i t_{j-1}}$. 
\medskip 

\noindent 
(a) We obtain 
	\begin{align} 
		\notag
		&\nquad\nquad\sum_{j:t_j\le T}\ \left\|\sum_{i=1}^\infty\Delta_{i,j}A\, S_i 
		\right\|^2 \\ 
		\notag
		&\nquad\le \sum_{j:t_j \le T}\left(\sum_{1\le i\le d}\big| 
		\Delta_{i,j}A\big|+\sum_{m=0}^\infty 2^{-\frac{m+2}{2}}\, \sum_{d2^m<i\le d2^{m+1}} 
		\big|\Delta_{i,j}A\big|\right)^2 \\ 
		\notag
		&\nquad\le \sum_{j:<t_j \le T}(t_j-t_{j-1})^2\left(d\cdot\max_{1\le i\le d}|G_i(0)| 
		\cdot\lambda_d+d\cdot\sum_{m=0}^\infty 2^{-\frac{m+2}{2}}\, \max_{d2^m+1< i\le d2^{m+1}} 
		|G_i(0)|\cdot\lambda_{d2^{m+1}}\right)^2 \\ 
		\label{eq:rc:p5}
		&\nquad\nquad\stack{|\tau|\to 0}{\lra} 0 
	\end{align} 
where we have used $|\Delta_{i,j}A|\le |G_i(0)|\, \lambda_i\cdot(t_j-t_{j-1})$ and the assumption 
that $\lambda_i>0$ is non-decreasing in $i\in \N$. We get part (a) of the lemma. \\ 
(b) Set $\Delta_{i,j}Z:=-\int_{u=t_{j-1}}^{t_j}\lambda_ie^{-\lambda_iu}W_i\left(e^{2\lambda_iu} 
-1\right)\, du$. With $\xi_{i,j}:=\sup_{u\in [t_{j-1},t_j]}{e^{-\lambda_iu}}\cdot W_i\left( 
e^{2\lambda_iu}-1\right)$ we have 
\begin{align*}
|\Delta_{i,j}Z|\le \xi_{i,j}\cdot\lambda_i\cdot (t_j-t_{j-1}).
\end{align*} 
By Lemma \ref{lem:rc:max:ou} we may find a constant $c>0$ independent of $j$ such that 
\begin{align*}
	E\left[\max_{ d2^m< i\le d2^{m+1}} \xi_{i,j}^k\right]\le c\left((\ln (\lambda_{d2^{m+1}}+1) 
	)^{\frac{k}{2}}+m^{\frac{k}{2}}\right),\quad m\in \{0,1,\ldots\},\ k=1,2,
\end{align*}
where for $\lambda_{d2^{m+1}} < 1$ we keep in mind the argumentation of Lemma \ref{lem:qv:var:1}.
We follow the ideas of the proof of Lemma \ref{lem:qv:var:1} (a) to obtain 
\begin{align}
	\notag
	&P\left(\sum_{j:t_j\le T}\ \left\|\sum_{i=1}^\infty \Delta_{i,j}Z\, S_i\right\|^2
	\ge\varepsilon\right)\le\frac1\varepsilon \sum_{j:t_j\le T}\ E\left[ 
	\left\|\sum_{i=1}^\infty \Delta_{i,j}Z\, S_i\right\|^2\right]\\
	\notag 
	&\quad \le\frac{d}{\eps}\ \sum_{j:t_j \le T} E\left[\bigg(\max_{0<i\le d}\big| 
	\Delta_{i,j}Z\big|+\sum_{m=0}^\infty 2^{-\frac{m+2}{2}}\max_{d2^m<i\le d2^{m+1}}\big| 
	\Delta_{i,j}Z\big|\bigg)^2\right]\\
	\notag 
	&\quad \le\frac{dD_1}{\eps}\sum_{j:t_j \le T} E\left[\bigg(\max_{0<i\le d}\xi_{i,j}  
	\cdot\lambda_d +\sum_{m=0}^\infty 2^{-\frac{m+2}{2}}\max_{d2^m<i\le d2^{m+1}} 
	\xi_{i,j}\cdot\lambda_{d2^{m+1}}\bigg)^2\right]\cdot (t_j-t_{j-1})^2  \\
	\label{eq:rc:p6}
	&\quad\le\frac{dD_2}{\eps}\ \bigg(\lambda_d+\sum_{m=0}^\infty 2^{-\frac{m+2}{2}}\cdot 
	\lambda_{d2^{m+1}}\cdot m^{\frac12}\bigg)^2\sum_{j:t_j \le T}(t_j-t_{j-1})^2\stack{|\tau| 
	\to 0}{\lra} 0 
\end{align} 
for some suitable constants $D_1>0$ and $D_2>0$ independent of $\lambda_i$, $i\in \N$, and 
$\tau^n$. En passant we have verified that the sum \eqref{eq:qv:ex:z} converges almost surely 
in $C_0([0,1]; \R^d)$. We have proved part (b) of the lemma.  
\end{proof}

\begin{proposition}
\label{qv:cla:prop:1} 
Assume \eqref{eq:rc:q1} and \eqref{eq:rc:q2}. Then $[X]_t=[Y]_t=t\theta$, $t\in [0,\infty)$ 
for both norms, $C_0([0,1];\R^d)$ and $L^1([0,1];\R^d)$. 
\end{proposition}
\begin{proof} 
Below we use the previous abbreviations $\Delta_{i,j}Y$, $\Delta_{i,j}Z$, and $\Delta_{i,j}A$. 
Let $T>0$. We have for both norms, in  $C_0([0,1];\R^d)$ and $L^1([0,1];\R^d)$,
	\begin{align}
		\notag
		&\left|\sum_{j:t_j<T}\left\|\sum_{i=1}^\infty  G_i(\lambda_it_{j})\cdot 
		S_i-\sum_{i=1}^\infty  G_i(\lambda_it_{j-1})\cdot S_i\right\|^2-\sum_{j:t_j<T} 
		\left\|\sum_{i=1}^\infty  \Delta_{i,j}YS_i\right\|^2\right|\\
		\notag
		&\quad\le 2\sum_{j:t_j<T}\left\|\sum_{i=1}^\infty \Delta_{i,j}YS_i\right\|\cdot 
		\left\|\sum_{i'=1}^\infty\left(\Delta_{i',j}ZS_{i'}+\Delta_{i',j}AS_{i'}\right)\right\| \\ 
		\label{eq:qv:terms:1}
		&\quad\quad +\sum_{j:t_j<T}\left\|\sum_{i=1}^\infty \left(\Delta_{i,j}Z 
		S_i+\Delta_{i,j}AS_i\right)\right\|^2.
	\end{align}
We study the expectation of the sum of the absolute values of the mixed terms in 
\eqref{eq:qv:terms:1} i.e.
	\begin{align*}
	\mathbf{E}(\tau):=\sum_{j:t_j<T}E\left[\left\| 
		\sum_{i=1}^\infty \Delta_{i,j}YS_i\right\|\cdot \left\|\sum_{i'=1}^\infty \left( 
		\Delta_{i',j}Z S_{i'}+\Delta_{i',j}AS_{i'}\right)\right\|\right],
	\end{align*}
	where the partition $\tau$ is given by $(t_j)_{j=1}^k$. By the Schwarz inequality applied as in 
\begin{align*} 
	\textstyle 
	E\left[\left\|\sum_i\right\|\cdot \left\|\sum_{i'}\right\|\right] 
	\le d^2\left( E\left[\left\|\sum_i\right\|_{C_0([0,1];\R^d)}^2\right]\right)^\hf\cdot 
	\left( E\left[\left\|\sum_{i'}\right\|_{C_0([0,1];\R^d)}^2\right]\right)^\hf 
\end{align*}
and calculations similar to \eqref{eq:rc:p2}, \eqref{eq:rc:p5}, and \eqref{eq:rc:p6} we verify 
\begin{align}
	\notag
	&\mathbf{E}(\tau)\le  D_3\left(\lambda_d^\hf+\sum_{m=0}^\infty 2^{-\frac{m}{2}}\cdot 
	\lambda_{d2^{m+1}}^\hf \cdot m^{\hf} \right)\times \\ 
	\notag
	&\quad\times\left(\big(\max_{0<i\le d}|G_i(0)|+1\big)\cdot\lambda_d+\sum_{m=0}^\infty 
	2^{-\frac{m}{2}}\big(\max_{d2^m<i\le d2^{m+1}}|G_i(0)|+m^{\hf}\big)\cdot \lambda_{d2^{m+1}} 
	\right)^2 \times\\
	&\quad\times \left(\sum_{j:t_j<T}(t_j-t_{j-1})^2\right)^\frac12
	\label{eq:rc:q3}
	\stack{|\tau|\to 0}{\lra} 0
\end{align}
for a suitable constant $D_3>0$. We obtain
\begin{align*}
	P\left(\sum_{j:t_j<T}\left\|\sum_{i=1}^\infty \Delta_{i,j}YS_i\right\|\cdot \left\|
	\sum_{i'=1}^\infty \left(\Delta_{i',j}Z S_{i'}+\Delta_{i',j}AS_{i'}\right)\right\| 
	\ge \eps \right)\le\frac{\mathbf{E}(\tau)}{\eps}.
\end{align*}
This and \eqref{eq:rc:q3} show that 
\begin{align}
	\label{eq:rc:ucp:2}
	&\sum_{j:t_j<t}\left\|\sum_{i=1}^\infty \Delta_{i,j}YS_i\right\|\cdot \left\|
	\sum_{i'=1}^\infty \left(\Delta_{i',j}Z S_{i'}+\Delta_{i',j}AS_{i'}\right)\right\| 
	\stack{|\tau|\to 0}{\lra} 0 
\end{align}
ucp on $t\in [0,\infty)$. Similarly, applying Lemma \ref{lem:qv:var:2} (a) and (b),
\begin{align}
	\label{eq:rc:ucp:2.5}
	\sum_{j:t_j<T}\left\|\sum_{i=1}^\infty\left(\Delta_{i,j}ZS_i+\Delta_{i,j}AS_i\right)\right\|^2
	\stack{|\tau|\to 0}{\lra} 0 
\end{align}
ucp on $t\in [0,\infty)$. Relations \eqref{eq:rc:ucp:2}, \eqref{eq:rc:ucp:2.5}, and 
\eqref{eq:qv:terms:1} prove the statement.
\end{proof}

For $r\in \R$ let $[r]$ denote the largest integer that does not exceed $r$. For $0\le a<\delta$  
and $t>0$ define 
\begin{align*}
	&\langle Y\rangle_a\equiv \langle Y\rangle_{a;t}(\delta)
	:=\sum_{k=1}^{\left[\frac{t-a}{\delta}\right]+1}\left(\|Y_{a+k\delta}-Y_{a+(k-1)\delta}\|^2
	-E\left[\|Y_{a+k\delta}-Y_{a+(k-1)\delta}\|^2\vphantom{l^1}\right]\vphantom{\dot{f}}\right) 
\end{align*} 
where as above $\|\, \cdot\, \|$ is the norm in either $C_0([0,1]; \R^d)$ or $L^1([0,1];\R^d)$. 

\begin{lemma}
\label{qv:reg:lem:1}
Suppose \eqref{eq:qv:lambda:1}. We have 
\begin{align} 
\label{eq:rc:ucp:3}
	\lim_{\delta\to 0}\sup_{0\le a<\delta}E\left[\Big(\langle Y\rangle_{a;t}(\delta)\Big)^2 
	\right]=0. 
\end{align}
\end{lemma}
\begin{proof} 
This is a particular case of Lemma \ref{lem:qv:var:1} (b).
\end{proof}

\begin{prop}
\label{qv:reg:prop:1}
Assume \eqref{eq:rc:q1} and \eqref{eq:rc:q2}.  Then we have
\begin{align*}
	\frac{1}{\delta} \int_0^t \left\|X_{s+\delta}-X_s\right\|^2 \, ds \stack 
	{\delta\to 0}{\lra}t\theta \quad \text{ucp on } t\in [0,\infty). 
\end{align*}
\end{prop}
\begin{proof} 
Let us use the same notation as in Lemma \ref{qv:reg:lem:1}. We obtain for all $t\in [0,\infty)$ 
and $\eps>0$
\begin{align}
	\notag
	&P\Bigg(\Bigg| \int_0^t \|Y_{s+\delta}-Y_s\|^2-E\left[\|Y_{s+\delta}-Y_s 
	\|^2\right] \, ds\Bigg|\ge \delta\eps\Bigg)\\
	\label{qv:reg:prop:1:eq:1}
	&\quad=P\Bigg(\left|\int_0^\delta \langle Y\rangle_a\, da\right| \ge \delta 
	\eps\Bigg)\le \frac{1}{\eps^2\delta^2}E\left[\left(\int_0^\delta \langle Y\rangle_a \, da\right)^2\right]. 
\end{align}
Furthermore, we have
\begin{align}
	\label{qv:reg:prop:1:eq:2}
	&E\left[\left(\int_0^\delta \langle Y\rangle_a\, da \right)^2\right]\le \delta 
	\int_0^\delta E\left[\langle Y\rangle_a^2\right]\, da\le \delta^2 
	\sup_{0\le a<\delta} E\left[\Big(\langle Y\rangle_{a;t}(\delta)\Big)^2\right].
\end{align}
As in the proof of Lemma \ref{lem:qv:var:1} (a), limit \eqref{eq:rc:p3}, it follows that 
$\frac1\delta E\left[\|Y_{s+\delta}-Y_s\|^2\right]\stack{\delta\to 0}{\lra}\theta$, 
$s\in [0,\infty)$. 
Relations \eqref{qv:reg:prop:1:eq:1} and \eqref{qv:reg:prop:1:eq:2} together with Lemma 
\ref{qv:reg:lem:1} show that $\frac1\delta\int_0^t \|Y_{s+\delta}-Y_s\|^2\, ds\stack{\delta 
\to 0}{\lra}\theta t$ in probability. Thanks to the fact that $[0,\infty)\ni t\to \frac1\delta 
\int_0^t \|Y_{s+\delta}-Y_s\|^2\, ds$ is increasing for all $\delta$, Lemma 3.1 in 
\cite{RussoVallois1999} implies even
\begin{align} 
\label{qv:reg:prop:1:eq:3}
	\frac1\delta\int_0^t \|Y_{s+\delta}-Y_s\|^2\, ds\stack{\delta\to 0}{\lra}\theta t\quad 
	\text{ucp on } [0,\infty). 
\end{align} 

Recalling $A_t=\sum_{i=1}^\infty e^{-\lambda_it}G_i(0)\cdot S_i$ and defining $(Z+A)_t:=Z_t+A_t$, 
$t\ge 0$, as in \eqref{eq:qv:terms:1} we may conclude 
\begin{align}
	\notag
	&\left|\frac1\delta\int_0^t\left\|X_{s+\delta}-X_s\right\|^2\, ds-\frac1\delta\int_0^t 
	\left\|Y_{s+\delta}-Y_s\right\|^2\, ds\right| \\ 
	\label{qv:prop:1:eq:1}
	&\quad\le\frac{2}{\delta}\int_0^t\left\|Y_{s+\delta}-Y_s\right\|\left\|(Z+A)_{s+\delta} 
	-(Z+A)_s\right\|\, ds+\frac1\delta\int_0^t\left\|(Z+A)_{s+\delta}-(Z+A)_s\right\|^2\, ds \\ 
	\notag
	&\quad=\frac{2}{\delta}\int_0^\delta\sum_{k=1}^{\left[ \frac{t-a}{\delta}\right]+1}
	\left\|Y_{a+k\delta}-Y_{a+(k-1)\delta}\right\|\left\|(Z+A)_{a+k\delta}-(Z+A)_{a+(k-1)\delta} 
	\right\|\, da \\ 
	\notag
	&\qquad+\frac1\delta\int_0^\delta\sum_{k=1}^{\left[ \frac{t-a}{\delta}\right]+1}\left\| 
	(Z+A)_{a+k\delta}-(Z+A)_{a+(k-1)\delta}\right\|^2\, da. 
\end{align} 
The two items on the right hand side tend to zero ucp on $t\in [0,\infty)$ by following 
the method of \eqref{qv:reg:prop:1:eq:1}-\eqref{qv:reg:prop:1:eq:3} with Markov's instead of 
Chebychev's inequality and applying the estimates of the proof of Proposition \ref{qv:cla:prop:1}. 
The claim follows. 
\end{proof}








\section{Tensor quadratic variation} 
\label{tv:section:1}

Let us use the notation of the previous section. In addition, let $\|\cdot\|_\pi$ denote the 
projective norm with respect to the tensor product $L^1([0,1];\R^d)\otimes L^1([0,1];\R^d)$, 
i.e. the norm in the Banach space $\lptp$. We recall that the algebraic tensor product $L^1([0,1];\R^d)\otimes L^1([0,1];\R^d)$ denotes all elements of the form 
$\sum_{i=1}^n x_i\otimes y_i$ where $x_i,y_i\in L^1([0,1];\R^d)$, $n\in \N$. For an element $u\in L^1([0,1];\R^d)\otimes L^1\left([0,1];\R^d\right)$ the projective norm is defined by
\begin{align*}
	\|u\|_\pi:=\inf \left\{\sum_{i=1}^n \|x_i\|_{L^1}\|y_i\|_{L^1}:u=\sum_{i=1}^n x_i\otimes y_i,\ \ x_i,y_i\in L^1([0,1];\R^d)\right\}.
\end{align*}
The space $\lptp$ is the completion of $L^1([0,1];\R^d)\otimes L^1([0,1];\R^d)$ with respect to the projective norm.
Recall that there is an isometric isomorphism $i:\lptp\to L^1([0,1]\times[0,1];\R^{d^2})$ 
given by 
\begin{align*}
(i\circ f\otimes g)(u,v)=(i\circ f\ptp g)(u,v) :=f(u)\otimes g(v),\quad f,g\in L^1([0,1];\R^d),\ 
(u,v)\in [0,1]^2. 
\end{align*} 
The reference measure $\mu$ on $([0,1]^2,{\cal B}([0,1]^2))$ is the Lebesgue measure. 
This also says that $f(u)\otimes g(v)$ is identified with $\left((f(u))_i\, (g(v))_j 
\right)_{i,j=1\ldots d}\, $. Here $\R^{d^2}$ is endowed with the norm $\|{\bf x}\|_{1,1} 
:=\sum_{i,j=1}^d|x_{ij}|$ where ${\bf x}=(x_{11},\ldots ,x_{dd})$, $x_{11},\ldots ,x_{dd}\in\R$. 
Similarly for the norm in $\R^d$ we denote $\|{\bf y}\|_{1}:=\sum_{i=1}^d|y_{i}|$, where ${\bf y}=(y_{1},\ldots ,y_{d})$, 
$y_{1},\ldots ,y_{d}\in\R$. We also introduce the notation $x^{\otimes^2} \equiv x\otimes x$ which will be used below. 
\medskip

In this section we aim to determine the tensor valued quadratic variation
\begin{align*}
	[X]^\otimes_t:=\lim_{\delta\to 0} \int_0^t \frac{(X_{u+\delta}-X_u)\otimes (X_{u+\delta} 
	-X_u)}{\delta} \,du
\end{align*}
in the ucp sense with respect to the $\pi$-norm. Below we take advantage of \cite{Girolami2014}, 
and \cite{GirolamiRusso2014}, however we would also like to refer to the classical work of \cite{MetivierPellaumail}. Let $\xi_i$, $i\in \N$, be independent 
standard normal random variables and define
\begin{align}
	\label{tv:eq:1}
	\Theta:=2E\left[\left(\sum_{i=1}^\infty \lambda_i^\hf \xi_i S_i \right)\otimes \left( 
	\sum_{i'=1}^\infty \lambda_{i'}^\hf \xi_{i'} S_{i'} \right)\right]
\end{align} 
provided that this expression exists in $\lptp$. To ensure compatibility with \cite{Girolami2014}, 
\cite{GirolamiRusso2014}, and \cite{JansonKaijser2015} all integrals with Banach space valued 
integrands, as for example in the definitions of $[X]^\otimes$ and $\Theta$, are Bochner integrals. 
\begin{prop}
\label{tv:prop:1} 
Suppose \eqref{eq:rc:q1} and \eqref{eq:rc:q2}, i.e. 
\begin{align*}
\sum_{m=0}^\infty 2^{-\frac{m}{2}}\cdot \lambda_{d2^{m+1}}\cdot\left(\max_{d2^m+1< i\le d2^{m+1}} 
|G_i(0)|+ m^{\hf}\right)<\infty. 
\end{align*} 
Then $\Theta$ is well-defined by \eqref{tv:eq:1} and we have in the norm of $\lptp$ 
\begin{align*}
	\frac{1}{\delta} \int_0^t \left(X_{s+\delta}-X_s\right)\otimes \left(X_{s+\delta}-X_s\right) 
	\, ds \stack {\delta\to 0}{\lra} t\Theta \quad \text{ucp on } t\in [0,\infty). 
\end{align*}
\end{prop}
\begin{proof} 
Using the independence of the $\xi_i$, $i\in \N$, we obtain 
\begin{align*} 
	&E\left[\left\|i\circ\left(\sum_{i=1}^\infty \lambda_i^\hf \xi_i S_i \right)\otimes \left( 
	\sum_{i'=1}^\infty \lambda_{i'}^\hf\xi_{i'} S_{i'} \right)\right\|_{L^1([0,1]^2;\R^{d^2})} 
	\right]\le \left(\sum_{i=1}^\infty \lambda_i^\hf\left(E[\xi_i^2]\right)^\hf\|S_i\|_{L^1([0,1]; 
	\R^d)} \right)^2 \\ 
	&\quad\le\left(\sum_{i=1}^d\lambda_i^\hf+\sum_{m=0}^\infty 2^{-\frac{m+4}{2}} 
	\lambda_{d2^{m+1}}^\hf\right)^2 <\infty,
\end{align*} 
where the second line follows from \eqref{eq:rc:q2}. Thus the existence and representation of $\Theta$ defined in \eqref{tv:eq:1} follows from 
\cite{JansonKaijser2015}, Theorem 10.2.

Without further reference, we will use the inequality $P\left(\sum_i \left|\zeta_i\right|\ge a 
\right)\le \sum_i P\left(\left|\zeta_i\right|\ge b_ia\right)$ several times in the proof. Here 
the $\zeta_i$ are arbitrary random variables, $b_i> 0$ with $\sum_i b_i=1$, and $a>0$.
\\[0.3cm]
\textit{Step 1: }
Let us introduce some of the important objects in the proof. First we recall that $X_t= 
Y_t+Z_t+A_t$, $t\ge 0$, where the individual items $Y$, $Z$, and $A$ are defined in the
beginning of Section \ref{qv:section:1}. Let
\begin{align*}
	\eta_t(\delta):=\frac{1}{\delta}\int_0^\delta (Y_s-Y_0)^{\otimes^2}\, ds 
	+\frac{1}{\delta}\int_0^{t}\left(Y_{(s+\delta)\wedge t}-Y_{s}\right)^{\otimes^2} \, ds 
	-t\Theta,\quad t\ge\delta\ge 0.
\end{align*} 
Among other things, this says 
\begin{align} 
	\notag 
	&\frac1\delta\int_0^{t}\left(Y_{s+\delta}-Y_{s}\right)^{\otimes^2}\, ds -t\Theta
	=\eta_t(\delta) +\frac1\delta\left(-\int_0^\delta (Y_s-Y_0)^{\otimes^2}\, ds
	+\int_{t-\delta}^t\left(Y_{s+\delta}-Y_t\right)^{\otimes^2} \, ds \right. \\
	\notag
	&\qquad+\left.\int_{t-\delta}^t\left(Y_{s+\delta}-Y_t\right)\otimes\left(Y_t-Y_s\right)\, ds
	+\int_{t-\delta}^t\left(Y_t-Y_s\right)\otimes \left(Y_{s+\delta}-Y_t\right)\, ds\right) \\
	\label{tv:prop:1:eq:-1} 
	&\quad=:\eta_t(\delta)+\zeta_t(\delta)\equiv \eta_t+\zeta_t. 
\end{align} 
We look at the second item of the right-hand side. 
Taking into consideration the above isometry and Fubini's theorem we get for $T\ge \delta$ and 
$\eps>0$ 
\begin{align} 
	\label{tv:prop:1:eq:0}
	&P\left(\sup_{0\le t\le T}\left\|\delta \zeta_t\right\|_\pi \ge \frac14\delta\eps\right)
	\le P\left(\sup_{0\le t\le T,\ 0\le\gamma\le\delta}\left\|Y_{t+\gamma}-Y_t 
	\right\|_{L^1([0,1];\R^d)}^2\ge\frac{1}{16}\eps\right)=:g(\delta;\eps).
\end{align} 
\textit{Step 2: }
We examine $g(\delta;\eps)$. According to L\'evy's characterization of Brownian motion we have
\begin{align*}
	Y_{t+\gamma}-Y_t=\sum_{i=1}^\infty \left((2\lambda_i)^\hf V_i(t+\gamma)S_i-(2\lambda_i)^\hf 
	V_i(t)S_i\right), 
\end{align*}
where $V_i$, $i\in \N$, are independent one-dimensional Wiener processes. 
We obtain 
\begin{align}
	\label{tv:prop:1:eq:7}
	g(\delta;\eps)&\le\sum_{i=1}^\infty 
	P\left(\sup_{0\le t\le T,\ 0\le\gamma\le\delta} \|S_i\|_{L^1}(2\lambda_i)^\hf 
	|V_i(t+\gamma)-V_i(t)| \ge \frac{b_i}{4}\eps^\hf\right)
\end{align}
where we choose  
\begin{align}
	\label{tv:prop:1:eq:8}
	b_i:=\frac{4(\ln (i+1))^\hf\cdot (2\lambda_i)^\hf\|S_i\|_{L^1}}{4\sum_{i'=1}^\infty( 
	\ln (i'+1))^\hf\cdot(2\lambda_{i'})^\hf\|S_{i'}\|_{L^1}}\equiv\frac{4(\ln (i+1))^\hf 
	\cdot(2\lambda_i)^\hf\|S_i\|_{L^1}}{D} \, ,\quad i\in {\mathbb N}.
\end{align} 
We note that the denominator $D$ is finite by $\|S_i\|_{L^1}\equiv\|S_i\|_{L^1([0,1];\R^d)} 
=2^{-\frac{3m-4}{2}}$ if $d2^m+1\le i\le d2^{m+1}$ for $m\in \{0,1,2,\ldots\}$ and hypothesis 
\eqref{eq:rc:q2}. Next we use the subsequent inequality which follows from Lemma 2.1 of 
\cite{ChenCsorgo2001} by choosing $\mu=1$ and scaling. There exists $C>0$ independent of $v$ 
and $\delta$ such that 
\begin{align*} 
P\left(\sup_{0\le t\le T,\ 0\le\gamma\le\delta}|V_i(t+\gamma)-V_i(t)|\ge v\delta^\hf 
\right)\le \frac{C}{\delta}\exp\left\{-\frac{v^2}{3}\right\}\, ,\quad v>0.  
\end{align*}
We obtain
 \begin{align}
	\label{tv:prop:1:eq:9}
 	&P\left(\sup_{0\le t\le T,\ 0\le\gamma\le\delta}(2\lambda_i)^\hf \|S_i\|_{L^1}\cdot 
	|V_i(t+\gamma)-V_i(t)|\ge\frac{b_i}{4}\eps^\hf\right)\le \frac{C}{\delta}\exp\left\{ 
	-\frac{\ln (i+1)\cdot\eps}{3D^2\delta} \right\} 
 \end{align}
for some $C>0$ independent of $\eps$, $i$, and $\delta$. Applying \eqref{tv:prop:1:eq:8} and 
\eqref{tv:prop:1:eq:9} to \eqref{tv:prop:1:eq:7} yields
\begin{align}
	\label{tv:prop:1:eq:12}
	P\left(\sup_{0\le t\le T,\ 0\le\gamma\le\delta}\left\|Y_{t+\gamma}-Y_t\right\|_{L^1([0,1]; 
	\R^d)}\ge\frac{1}{4}\eps^\hf\right)\le\frac{C}{\delta}\sum_{i=1}^\infty\exp\left\{-\frac{\ln 
	(i+1)\cdot\eps}{3D^2\delta}\right\}.
\end{align}
It follows that 
\begin{align} 
	\label{tv:prop:1:eq:5}
	g(\delta;\eps) \stack{\delta\to 0}{\lra}0
\end{align} 
for all $\eps>0$. Using similar arguments we get 
\begin{align}
	\notag
	&P\left( \sup_{0\le t\le \delta}\left\|\int_0^{t}\left(Y_{s+\delta}-Y_{s}\right)^{\otimes^2} 
	\, ds -\delta t\Theta\right\|_\pi\ge \frac12\delta\eps\right)\\
	\label{tv:prop:1:eq:6}
	&\qquad\le P\left( \sup_{0\le s,t\le \delta}\left\|Y_{t}-Y_{s}\right\|_{L^1([0,1];\R^d)}^2+ 
	\delta\left\|\Theta\right\|_\pi \ge\frac12 \eps\right)=:h(\delta;\eps) \stack{\delta\to 0} 
	{\lra}0.
\end{align}
\textit{Step 3: }
Taking our attention back to $\eta_t$, for $t>u\ge\delta$ we note that
\begin{align*}
	&\eta_t=\frac{1}{\delta}\int_0^\delta (Y_s-Y_0)^{\otimes^2}\, ds 
	+\frac{1}{\delta}\int_0^{u-\delta} \left(Y_{s+\delta}-Y_{s}\right)^{\otimes^2} \, ds 
	+\frac{1}{\delta}\int_{u-\delta}^u (Y_u-Y_s)^{\otimes^2}\, ds \\ 
	&\qquad+\frac{1}{\delta}\int_{u-\delta}^u\left(Y_{(s+\delta)\wedge t}-Y_{u}\right)^{\otimes^2} 
	\, ds +\frac{1}{\delta}\int_{u-\delta}^u\left(Y_{(s+\delta)\wedge t}-Y_{u}\right)\otimes (Y_u 
	-Y_s)\, ds \\
	&\qquad+\frac{1}{\delta}\int_{u-\delta}^u(Y_u-Y_s)\otimes \left(Y_{(s+\delta)\wedge t}-Y_{u} 
	\right)\, ds+\frac{1}{\delta}\int_{u}^t\left(Y_{(s+\delta)\wedge t}-Y_{s}\right)^{\otimes^2}\, ds 
	- t\Theta \\ 
	&\quad=\eta_u+\frac{1}{\delta}\int_{u-\delta}^u\left(Y_{(s+\delta)\wedge t}-Y_{u}\right)^{ 
	\otimes^2}\, ds+\frac{1}{\delta}\int_{u-\delta}^u\left(Y_{(s+\delta)\wedge t}-Y_{u}\right) 
	\otimes (Y_u-Y_s)\, ds \\
	&\qquad+\frac{1}{\delta}\int_{u-\delta}^u (Y_u-Y_s)\otimes\left(Y_{(s+\delta)\wedge t}-Y_{u} 
	\right) \, ds+\frac{1}{\delta}\int_{u}^t\left(Y_{(s+\delta)\wedge t}-Y_{s}\right)^{\otimes^2} 
	\, ds - (t-u)\Theta. 
\end{align*}
By the definition of $\Theta$ and the above isomorphism it follows that $E[\eta_t|\F_u]= 
E[\eta_t-\eta_u|\F_u]+\eta_u=\eta_u$, $t>u\ge\delta$. In other words, $\eta_t$, $t\ge \delta$, 
is a martingale. For well-definiteness see again \cite{JansonKaijser2015}, Theorem 10.2.
\medskip

Using \eqref{tv:prop:1:eq:-1}, \eqref{tv:prop:1:eq:0}, \eqref{tv:prop:1:eq:6}, and Doob's 
inequality we obtain for $T\ge \delta$ and $\eps>0$
\begin{align}
	\notag
	&P\left( \sup_{0\le t\le T}\left\|\int_0^{t}\left(Y_{s+\delta}-Y_{s}\right)^{\otimes^2}\, ds 
	-\delta t\Theta\right\|_\pi\ge \delta\eps\right)\\
	\notag
	&\quad\le P\left( \sup_{0\le t\le \delta}\left\|\int_0^{t}\left(Y_{s+\delta}-Y_{s} 
	\right)^{\otimes^2}\, ds -\delta t\Theta\right\|_\pi\ge\frac12\delta\eps\right) \\
	\notag
	&\qquad+P\left( \sup_{\delta\le t\le T}\left\|\int_0^{t}\left(Y_{s+\delta}-Y_{s} 
	\right)^{\otimes^2}\, ds -\delta t\Theta\right\|_\pi\ge\frac12\delta\eps\right)\\
	\notag 
	&\quad\le h(\delta;\eps)+P\left(\sup_{\delta\le t\le T}\left\|\delta\eta_t\right\|_\pi\ge 
	\frac14\delta\eps\right)+g(\delta;\eps) \\ 
	\notag
	&\quad\le h(\delta;\eps)+\frac{16}{\eps^2\delta^2}E\left[\left\|\delta\eta_T\right\|_\pi^2 
	\right]+g(\delta;\eps) \\ 
	\label{tv:prop:1:eq:10}
	&\quad\le \frac{16}{\eps^2\delta^2} E\left[\left\|\left(\int_0^{T} \left(Y_{s+\delta}-Y_{s} 
	\right)^{\otimes^2} \, ds-\delta T\Theta\right)-\delta\zeta_T\right\|_\pi^2\right] 
	+g(\delta;\eps)+h(\delta;\eps). 
\end{align}
Next we introduce
\begin{align*}
	\langle Y\rangle_a^\otimes\equiv \langle Y\rangle_{a,T}^\otimes(\delta)
	&:=\sum_{k=1}^{\left[\frac{T-a}{\delta}\right]+1} \left[\left(Y_{a+k\delta}-Y_{a+(k-1)\delta} 
	\right)^{\otimes^2}-\delta\Theta \right] 
\end{align*}
and note that $\langle Y\rangle_a^\otimes$ is almost surely an element of $\in\lptp$. From 
\eqref{tv:prop:1:eq:10} we obtain 
\begin{align}
	\notag
	&P\left( \sup_{0\le t\le T}\left\|\int_0^{t}\left(Y_{s+\delta}-Y_{s}\right)^{\otimes^2}\, ds 
	-\delta t\Theta\right\|_\pi\ge \delta\eps\right)\\
	\notag
	&\quad= \frac{32}{\eps^2\delta^2} E\left[\left\| \int_0^\delta \langle Y\rangle_a^\otimes 
	\, da\right\|_\pi^2\right]+\frac{32}{\eps^2\delta^2} E\left[\left\|\delta\zeta_T\right\|_\pi^2 
	\right]+g(\delta;\eps)+h(\delta;\eps)\\
	\label{tv:prop:1:eq:1}
	&\quad\le \frac{32}{\delta\eps^2}\int_0^\delta E\left[\left\|\langle Y\rangle_a^\otimes 
	\right\|_\pi^2\right]\, da+\frac{32}{\eps^2\delta^2} E\left[\left\|\delta\zeta_T\right\|_\pi^2 
	\right]+g(\delta;\eps)+h(\delta;\eps).
\end{align} 
\textit{Step 4: } 
In this step we examine the expression $E\left[\left\|\langle Y\rangle_a^\otimes \right\|_\pi^2\right]$. We do this by first applying the 
isometry $i:\lptp\to L^1([0,1]\times[0,1];\R^{d^2})$, followed by the Schwarz inequality, and Fubini's theorem 
to obtain
\begin{align}
	\notag
	&\nquad E\left[\left\| \langle Y\rangle_a^\otimes \right\|_\pi^2\right]
	= E\left[\int_{([0,1]^2)^2}\left\| \left(i\circ \langle Y\rangle^\otimes_a\right)(u,v) 
	\right\|_{1,1}\cdot\left\|\left(i\circ\langle Y\rangle^\otimes_{a'}\right)(u',v')\right\|_{1,1}\, d 
	\mu(u,v)d\mu(u',v')\right]\\
	\notag
	&=\int_{([0,1]^2)^2} E\left[ \left\|\left(i\circ \langle Y\rangle^\otimes_a\right)(u,v) 
	\right\|_{1,1}\cdot\left\|\left(i\circ\langle Y \rangle^\otimes_{a'}\right)(u',v')\right\|_{1,1} 
	\vphantom{\dot{f}}\right]\, d\mu(u,v)d\mu(u',v')\\
	\label{tv:prop:1:eq:2}
	&\le\left(\int_{[0,1]^2} \left(E\left[\left\|\left(i\circ \langle Y\rangle^\otimes_a 
	\right)(u,v)\vphantom{\dot{f}}\right\|_{1,1}^2\right] \right)^\hf d\mu(u,v)\right)^2. 
\end{align} 
Taking into consideration the independence of the increments of $Y$ we get
\begin{align}
	\notag
	&E\left[\left\| \left(i\circ \langle Y\rangle^\otimes_a\right)(u,v)\right\|_{1,1}^2 \right]
	=\sum_{k=1}^{\left[ \frac{t-a}{\delta}\right]+1} E\left[\left\|\left(i\circ\left(\left( 
	Y_{a+k\delta}-Y_{a+(k-1)\delta}\right)^{\otimes^2}-\delta\Theta\right)\right)\left(u,v\right) 
	\right\|_{1,1}^2\right]\\
	\notag
	&\quad=\sum_{k=1}^{\left[\frac{t-a}{\delta}\right]+1} E\left[\vphantom{\left.\vphantom{\dot{f}}\right\|_{1,1}^2}\left\|\vphantom{\dot{f}} 
	\left(Y_{a+k\delta}-Y_{a+(k-1)\delta}\right)(u)\otimes\left(Y_{a+k\delta}-Y_{a+(k-1)\delta} 
	\right)(v) \right.\right.\\
	\notag
	&\qquad- \left.\left.E\left[\left(Y_{a+k\delta}-Y_{a+(k-1)\delta}\right)(u)\otimes\left( 
	Y_{a+k\delta}-Y_{a+(k-1)\delta}\right)(v)\right]\vphantom{\dot{f}}\right\|_{1,1}^2\right] \\
	\notag
	&\quad\le\sum_{k=1}^{\left[ \frac{t-a}{\delta}\right]+1} E\left[\left\|\vphantom{\dot{f}} 
	\left(Y_{a+k\delta}-Y_{a+(k-1)\delta}\right)(u)\otimes\left(Y_{a+k\delta}-Y_{a+(k-1)\delta} 
	\right)(v)\vphantom{\dot{f}}\right\|_{1,1}^2\right]. 
\end{align}
It follows from the Schwarz inequality that
\begin{align}
	\notag
	&E\left[\left\| \left(i\circ \langle Y\rangle^\otimes_a\right)(u,v)\right\|_{1,1}^2 \right]
	\\
	\label{tv:prop:1:eq:3}
	&\quad\le\sum_{k=1}^{\left[\frac{t-a}{\delta}\right]+1}\left( E\left[\left\|\left( 
	Y_{a+k\delta}-Y_{a+(k-1)\delta}\right)(u)\right\|_1^4\right]\right)^\hf \cdot \left( 
	E\left[\left\|\left(Y_{a+k\delta}-Y_{a+(k-1)\delta}\right)(v)\right\|_1^4\right]\right)^\hf.
\end{align}
We recall from the proof of Lemma \ref{lem:qv:var:1} that $Y_{a+k\delta}-Y_{a+(k-1)\delta}= 
\sum_{i=1}^\infty\Delta_{i,a+k\delta}Y\cdot S_i$ where $\Delta_{i,a+k\delta}$ is $N(0,2\lambda_i 
\delta)$-distributed, $i\in \N$, and thus by relation \eqref{eq:rc:p2} and hypothesis 
\eqref{eq:rc:q2} 
\begin{align}
\label{tv:prop:1:eq:4}
	E\left[\left\|\left(Y_{a+k\delta}-Y_{a+(k-1)\delta}\right)(u)\right\|_1^4\right]\le C\delta^2 
\end{align}
for some $C>0$ independent of $u\in [0,1]$. By \eqref{tv:prop:1:eq:-1} and similar calculations 
it follows that 
\begin{align}
\label{tv:prop:1:eq:13}
	 E\left[\left\| \zeta_T\right\|_\pi^2\right] \stack{\delta\to 0} {\lra}0.
\end{align}
Now \eqref{tv:prop:1:eq:1} together with 
\eqref{tv:prop:1:eq:5}, \eqref{tv:prop:1:eq:6}, \eqref{tv:prop:1:eq:13}, and 
\eqref{tv:prop:1:eq:2}-\eqref{tv:prop:1:eq:4} imply 
\begin{align*}
	\int_0^t \left(Y_{s+\delta}-Y_{s}\right)^{\otimes^2}\, ds \stack{\delta\to 0}{\lra}t\Theta 
	\quad \text{ucp on } [0,\infty) 
\end{align*}
in the norm of $\lptp$. 
\medskip 

\noindent
\textit{Step 5: } 
It remains to show that
\begin{align*}
	\left\|\frac1\delta\int_0^t \left(X_{a+\delta}-X_{a}\right)^{\otimes^2}\, da-\frac1\delta 
	\int_0^t \left(Y_{a+\delta}-Y_{a}\right)^{\otimes^2}\, da\right\|_\pi \stack{\delta\to 0} 
	{\lra} 0\quad \text{ucp on } [0,\infty). 
\end{align*}
Using the relation $X_t=Y_t+Z_t+A_t$, as defined in Section \ref{qv:section:1}, we obtain
\begin{align*}
	&\nquad\left\|\frac1\delta\int_0^t \left(X_{s+\delta}-X_{s}\right)^{\otimes^2}-\left(Y_{s+\delta} 
	-Y_{s}\right)^{\otimes^2}\, ds\right\|_\pi\\
	&\nquad=\left\|\frac1\delta\int_0^t \left(Y_{s+\delta}+Z_{s+\delta}+A_{s+\delta}-Y_{s}-Z_s-A_s 
	\right)^{\otimes^2}-\left(Y_{s+\delta}-Y_{s}\right)^{\otimes^2}\, ds\right\|_\pi\\
	&\nquad=\left\|\frac1\delta\int_0^t \left(Y_{s+\delta}-Y_s\right)\otimes \left(Z_{s+\delta} 
	+A_{s+\delta}-Z_s-A_s\right)+ \left(Z_{s+\delta}+A_{s+\delta}-Z_s-A_s\right)\otimes 
	\left(Y_{s+\delta}-Y_s\right)\right.\\
	&\quad\left.+\left(Z_{s+\delta}+A_{s+\delta}-Z_s-A_s\right)^{\otimes^2} \, ds\right\|_\pi.
\end{align*}
Using the isometry of $i:\lptp\to L^1([0,1]^2;\R^{d^2})$ and the triangle inequality we 
get the estimate   
\begin{align*}
	&\left\|\frac1\delta\int_0^t \left(X_{s+\delta}-X_{s}\right)^{\otimes^2}-\left(Y_{s+\delta} 
	-Y_{s}\right)^{\otimes^2}\, ds\right\|_\pi \\
	&\quad\le\frac{1}{\delta}\int_0^t 2\left\|Y_{s+\delta}-Y_s\right\|_{L^1} 
	\left\|(Z+A)_{s+\delta}-(Z+A)_s\right\|_{L^1}+\left\|(Z+A)_{s+\delta}-(Z+A)_s 
	\right\|_{L^1}^2\, ds 
\end{align*}
with $\|\cdot\|_{L^1}$ abbreviating the norm in $L^1([0,1];\R^d)$. This is precisely the 
expression \eqref{qv:prop:1:eq:1} of the proof of Proposition \ref{qv:reg:prop:1} (b). The 
proof of Proposition \ref{qv:reg:prop:1} (b) shows that under the assumptions \eqref{eq:rc:q1} 
and \eqref{eq:rc:q2} this expression tends to $0$ ucp as $\delta\to 0$. The claim follows.
\end{proof}







\section{It\^o's formula} 
\label{it:section:1}

According to Lemma \ref{lem:qv:var:1} (b) and relation \eqref{eq:rc:p6}, the processes $Y$ and $Z$ are quadratically
integrable provided that \eqref{eq:rc:q1} and \eqref{eq:rc:q2} hold, choose $\tau=\{0=t_0,t_1=T\}$ for this.
Thus, depending on (quadratic) integrability of $\sum_{i=1}^\infty G_i(0)S_i$, the process
$X$ is a (quadratically) integrable semimartingale with decomposition $X=Y+(Z+A)$, $Y$ being
the martingale part. In order to establish an It\^o formula for the process $X$, one could
think of applying the It\^o formula in Banach spaces given by \cite{GyongyKrylov1981}, if possible. However,
having established the tensor quadratic variation in Section 4, it is more natural and more
direct to take advantage of the It\^o formula in Banach spaces corresponding to the stochastic
calculus of regularization, see \cite{Girolami2014} and \cite{GirolamiRusso2014}. 
\medskip

To ease the notation in this section we denote $B:=L^1([0,1];\R^d)$. Below we use the pairing dualities 
${}_{B^\ast}\langle \ \cdot \ ,\ \cdot \ \rangle_B$ and ${}_{(B\ptp B)^\ast}\langle\ \cdot\ , 
\ \cdot\ \rangle_{(B\ptp B)^{\ast\ast}}$ in the sense and notation of \cite{Girolami2014} and 
\cite{GirolamiRusso2014}. Using the standard identification of the dual spaces, we recall that for $f\in L^1([0,1];\R^d)=B$, $g^\ast\in B^\ast$, and 
some representing element $g\in L^\infty([0,1];\R^d)\cong B^\ast$ of $g^\ast$ we have
\begin{align*}
	{}_{B^\ast}\langle g^\ast ,f \rangle_B=\int_0^1 f(u)\cdot g(u)\, du.
\end{align*}
Furthermore we restrict ${}_{(B\ptp B)^\ast}\langle\ \cdot\ ,F \rangle_{(B\ptp B)^{\ast\ast}}$ 
to $F\in B\ptp B\cong L^1([0,1]^2;\R^{d^2})$. To discern the different scalar products we will use the symbol $\sbt$ to denote the scalar 
product in $\R^{d^2}$. For $G^\ast\in (B\ptp B)^\ast$ and some representing element $G\in 
L^\infty([0,1]^2;\R^{d^2}) \cong (B\ptp B)^\ast$ the pairing duality becomes
\begin{align*}
	{}_{(B\ptp B)^\ast}\left\langle G^\ast,F \right\rangle_{(B\ptp B)^{\ast\ast}}=\int F(u,v)\sbt 
	G(u,v)\, dudv.
\end{align*}

\begin{definition}
	Let $(X_t)_{t\in [0,T]}$ and $(Y_t)_{t\in [0,T]}$ be continuous $B$-valued, respectively 
	$B^\ast$-valued stochastic processes. The \emph{forward integral of $Y$ with respect to $X$} 
	denoted by $\int_0^t  {}_{B^\ast}\langle Y_s ,dX_s \rangle_B$ is defined as the limit
	\begin{align*}
		\int_0^t  {}_{B^\ast}\langle Y_s ,dX_s \rangle_B:=\lim_{\eps\to 0} \int_0^t 
		{\vphantom{\bigg\langle}}_{B^\ast}\left\langle Y_s ,\frac{X_{s+\eps}-X_s}{\eps} 
		\right\rangle_B\, ds
	\end{align*}
	in probability if it exists.
\end{definition}
\begin{definition}
Let $F$ be a mapping $F:[0,T]\times B \to \R$. We say that $F$ is of \emph{Fr\'echet class 
$C^{1,2}$} (in symbols $F\in C^{1,2}$) if $F$ is one time continuously Fr\'echet differentiable 
and two times continuously Fr\'echet differentiable in the second argument. That is, denoting 
the Fr\'echet derivative with respect to the second variable by $D$, for every $t\in[0,T]$ we 
have $D F(t,\cdot): B\to B^\ast$ and $D^2 F(t,\cdot): B\to (B\ptp B)^\ast$ continuously.
\end{definition}
\begin{theorem} 
	\label{it:theorem:1}
	Suppose \eqref{eq:rc:q1} and \eqref{eq:rc:q2}, i.e. 
\begin{align*}
\sum_{m=0}^\infty 2^{-\frac{m}{2}}\cdot \lambda_{d2^{m+1}}\cdot\left(\max_{d2^m+1< i\le d2^{m+1}} 
|G_i(0)|+ m^{\hf}\right)<\infty. 
\end{align*}Let $(X_t)_{t\in [0,T]}$ be given by \eqref{rc:intro:eq:1}, i.e. 
\begin{align*}
	X_t=\sum_{i=1}^\infty  G_i(\lambda_it)\cdot S_i,\quad t\ge 0.
\end{align*}
Furthermore denote $B=L^1([0,1];\R^d)$, and let 
	$F\in C^{1,2}$. Then
	\\[0.3cm] (a) For every $t\in [0,T]$ the forward integral $\int_0^t  {}_{B^\ast}\langle 
	DF(s,X_s) ,dX_s \rangle_B$ exists.
	\\[0.3cm] (b) We have the It\^o formula
	\begin{align*}
		F(t,X_t)&=F(0,X_0)+\int_0^t \frac{\partial}{\partial s} F(s,X_s)\, ds+\int_0^t 
		{}_{B^\ast}\langle DF(s,X_s),dX_s\rangle_B \\
		&\quad+\hf \int_0^t {\vphantom{\big(}}_{(B\ptp B)^\ast}\left\langle D^2F(s,X_s), 
		\Theta\right\rangle_{(B\ptp B)^{\ast\ast}}\, ds,
	\end{align*}	
	where $\Theta$ is given by \eqref{tv:eq:1}.
\end{theorem}
\begin{proof}
(a) Due to the existence of a scalar and tensor quadratic variation for $X_t$ we may apply 
Proposition 3.15 of \cite{GirolamiRusso2014} and Theorem 6.3 of \cite{Girolami2014} from which 
the statement follows.
\\[0.3cm] (b) This is an immediate consequence of Proposition \ref{tv:prop:1} above, and 
Theorem 6.3 of \cite{Girolami2014} together with Remark 6.2 of \cite{Girolami2014}.
\end{proof}

Now let us specify the It\^o formula to cylindrical functions $F$ of type $F(s,\gamma) 
=f(s\, ;\langle S_1,\gamma\rangle,\ldots, \langle S_k,\gamma\rangle)$, $f\in C_0^\infty\left( 
\R^{k+1}\right)$, $s\ge 0$, $\gamma\in C_0([0,1];{\mathbb R}^d)$. Here, $S_i$ and $\langle S_i, 
\gamma \rangle$ are given by $S_i(s):=\int_0^s g_i(u)\, du$ and $\langle S_i,\gamma \rangle:= 
\int_0^1 g_i(u)\, d\gamma_u$, cf. \eqref{intro:eq:1}. Since such functions $F$ are discontinuous 
if $\gamma$ is considered as an element belonging to $L^1([0,1];{\mathbb R}^d)$ we cannot directly 
apply Theorem \ref{it:theorem:1} (b). 

For $i=d(2^m+k-1)+j$, $m\in \{0,1,\ldots\}$, $k\in \{1,\ldots,2^m\}$, $j\in \{1,\ldots,d\}$, let 
$g_i^{(\alpha)}$, $\alpha\in (0,1)$, be an element of $C_0^\infty([0,1];\R^d)$ such that $g_i^{ 
(\alpha)}(s)\cdot e_{j'}=0$, $s\in [0,1]$, for $\{1,\ldots,d\}\ni j'\neq j$. Furthermore, 
define the $\R^d$-valued signed measures $\mu_i^{(\alpha)}$ and $\mu$ on $([0,1],\B([0,1])$ by 
$\mu_i^{(\alpha)}((a,b]):=g_i^{(\alpha)}(b)-g_i^{(\alpha)}(a)$ and $\mu_i((a,b]):=g_i(b)-g_i(a)$. 
Suppose that $\mu_i^{(\alpha)}$ converges to $\mu_i$ in the weak$^\ast$-topology as $\alpha\to 0$ 
and that $\|\mu_i^{(\alpha)}\|_v\le \|\mu_i\|_v$, $\alpha\in (0,1)$, where $\|\cdot\|_v$ denotes 
the total variation.

Now let us study cylindrical functions $F^{(\alpha)}$ of type $F^{(\alpha)}(s,\gamma)=f(s\, ; 
\langle S^{(\alpha)}_1,\gamma_s\rangle,\ldots,\langle S^{(\alpha)}_k,\gamma_s\rangle)$ where 
$f$, $s$, and $\gamma$ are as above and $S^{(\alpha)}_i(s):=\int_0^s g^{(\alpha)}_i(u)\, du$ as 
well as $\langle S^{(\alpha)}_i,\gamma \rangle:=\int_0^1 g^{(\alpha)}_i(u)\, d\gamma_u$. 

It follows that
\begin{align}
\label{it:eq:1}
	\notag
	&DF^{(\alpha)}(s,\gamma)(h)
	=\sum_{i=1}^k \frac{\partial}{\partial x_i}f\left(s\, ;\langle S^{(\alpha)}_1,\gamma\rangle, 
	\ldots,\langle S^{(\alpha)}_k,\gamma\rangle\right)\langle S^{(\alpha)}_i,h \rangle \\ 
	\notag
	&\quad=\sum_{i=1}^k \frac{\partial}{\partial x_i}f\left(s\, ;\langle S^{(\alpha)}_1,\gamma 
	\rangle,\ldots,\langle S^{(\alpha)}_k,\gamma\rangle\right) \int_0^1 \left(-g_i^{(\alpha)} 
	\right)'\cdot h \,dt\\ 
	&\quad={}_{B^\ast}\langle DF^{(\alpha)}(s,\gamma),h\rangle_B\vphantom{\sum_{i=1}^k}
\end{align}
where this chain of equations is true for all $h\in \{\tilde{h}:\tilde{h}\in C([0,1];\R^d),\ 
h(0)=h(1)=0\}$. By the well-definiteness of the second line for all $h\in L^1([0,1];\R^d)$ we 
may continuously extend the left-hand side as well as the right-hand side in $L^1([0,1];\R^d)$ 
to all $h\in L^1([0,1];\R^d)$. Furthermore, for $\Phi\in L^1([0,1]^2;\R^{d^2})$ introduce 
\begin{align*}
	\tlangle{ (S^{(\alpha)}_i(\cdot))\otimes(S^{(\alpha)}_{i'}(\cdot)), \Phi }:=\int_{[0,1]^2} 
	g_i^{(\alpha)}(u)\otimes g_{i'}^{(\alpha)}(v) \sbt d\Phi (u,v). 
\end{align*}
In the same sense as above we have
\begin{align}
\label{it:eq:2} 
	\notag
	&\nquad\nquad D^2F^{(\alpha)}(s,\gamma)(H)
	=\sum_{i,i'=1}^k \frac{\partial^2}{\partial x_ix_{i'}}f\left(s\, ;\langle S^{(\alpha)}_1, 
	\gamma\rangle,\ldots,\langle S^{(\alpha)}_k,\gamma\rangle\right)\tlangle{(S^{(\alpha)}_i 
	(\cdot))\otimes(S^{(\alpha)}_{i'}(\cdot)),i\circ H}\\
	\notag
	&\nquad=\sum_{i,i'=1}^k\frac{\partial^2}{\partial x_ix_{i'}}f\left(s\, ;\langle S^{(\alpha)}_1, 
	\gamma\rangle,\ldots,\langle S^{(\alpha)}_k,\gamma\rangle\right) \int_{[0,1]^2} g_i^{(\alpha)} 
	(u)\otimes g_{i'}^{(\alpha)}(v) \sbt d(i\circ H)(u,v)\\
	\notag
	&\nquad=\sum_{i,i'=1}^k\frac{\partial^2}{\partial x_ix_{i'}}f\left(s\, ;\langle S^{(\alpha)}_1, 
	\gamma\rangle,\ldots,\langle S^{(\alpha)}_k,\gamma\rangle\right)\int_{[0,1]^2}\left( 
	g_i^{(\alpha)}(u)\right)'\otimes \left(g_{i'}^{(\alpha)}(v)\right)' \sbt (i\circ H)(u,v) 
	\ dudv\\
	&\nquad={}_{(B\ptp B)^\ast}\left\langle D^2F^{(\alpha)}(s,\gamma),H \right\rangle_{(B\ptp 
	B)^{\ast\ast}},\quad H\in  L^1([0,1];\R^{d})\ptp L^1([0,1];\R^{d}).\vphantom{\sum_{i=1}^k}
\end{align}

For the next lemma recall the decomposition $X_t=Y_t+Z_t+A_t$, $t\ge 0$, introduced in Section 
\ref{qv:section:1}.
\begin{lemma}
	\label{it:lem:1}
Suppose \eqref{eq:rc:q1} and \eqref{eq:rc:q2}. Let $i\in\N$ and $T>0$. 
\\
(a) For the quadratic variation $\left[\langle S^{(\alpha)}_i,Y_\cdot\rangle 
\right]_t$ of $\langle S^{(\alpha)}_i,Y_t\rangle$ at $t>0$ it holds that 
\begin{align*}
	\sup_{\alpha\in (0,1)}E\left[\left[\langle S^{(\alpha)}_i,Y_\cdot\rangle\right]_t\right] 
	<\infty. 
\end{align*} 
(b) For the total variation $\left|\langle S^{(\alpha)}_i,Z_\cdot+A_\cdot\rangle \right|_t$ of 
$\langle S^{(\alpha)}_i,Z_t+A_t\rangle$ at $t>0$ it holds that 
\begin{align*}
	\sup_{\alpha\in (0,1)}E\left[\left|\langle S^{(\alpha)}_i,Z_\cdot+A_\cdot\rangle\right|_t 
	\right] <\infty. 
\end{align*} 
(c) It holds that $\langle S^{(\alpha)}_i,X_t\rangle \stack{\alpha\to 0}{\lra} \left\langle S_i, 
X_t\right\rangle$ uniformly on $t\in [0,T]$ almost surely. 
\end{lemma}
\begin{proof} 
In the following proof we abbreviate $\|\cdot\|_{C_0([0,1];\R^d)}=\|\cdot\|_{C}$ and $\|\cdot 
\|_{L^1([0,1];\R^d)}=\|\cdot\|_{L^1}$. \\ 
(a) Using the partition of $[0,T]$ introduced in Section \ref {qv:section:1}, we have 
\begin{align*}
	&\sup_{\alpha\in (0,1)}E\left[\left[\langle S^{(\alpha)}_i,Y_\cdot\rangle\right]_T\right]
	=\sup_{\alpha\in (0,1)}E\left[\lim_{n\to\infty}\sum_{j:t_j\le T}\langle S^{(\alpha)}_i, 
	Y_{t_j}-Y_{t_{j-1}}\rangle^2\right]\\
	&\quad=\sup_{\alpha\in (0,1)}E\left[\lim_{n\to\infty}\sum_{j:t_j\le T}\left(\int_0^1\left( 
	-g_i^{(\alpha)}\right)'(u)\cdot\left(Y_{t_j}-Y_{t_{j-1}}\right)(u)\,du\right)^2\right]\\ 
	&\quad\le\sup_{\alpha\in (0,1)}E\left[\lim_{n\to\infty}\sum_{j:t_j\le T}\left\|\left( 
	g_i^{(\alpha)}\right)'(u)\right\|^2_{L^1}\cdot\left\|Y_{t_j}-Y_{t_{j-1}}\right\|^2_{C}\right] 
    \\ 
	&\quad\le\sup_{\alpha\in (0,1)}\left\|\left(g_i^{(\alpha)}\right)'(u)\right\|^2_{L^1}\cdot 
	E\left[\lim_{n\to\infty}\sum_{j:t_j\le T}\left\|Y_{t_j}-Y_{t_{j-1}}\right\|^2_{C}\right].
\end{align*} 
This proves part (a) of the lemma since $\|(g_i^{(\alpha)})'(u)\|_{L^1}=\|\mu_i^{(\alpha)}\|_v 
\le \|\mu_i\|_v$, $\alpha\in (0,1)$, by hypothesis. Furthermore the second term is finite by 
Lemma \ref{lem:qv:var:1} (a). \\ 
(b) We have
\begin{align*}
	&\nquad\nquad\sup_{\alpha\in (0,1)} E\left[\left|\langle S_i^{(\alpha)},Z_\cdot+A_\cdot \rangle\right|_t 
	\right]
	\le \sup_{\alpha\in (0,1)} E\left[\left|-\sum_{j=1}^\infty \int_{u=0}^\cdot\lambda_j 
	e^{-\lambda_ju}W_j\left(e^{2\lambda_ju}-1\right)\, du \cdot \langle S_i^{(\alpha)},S_j\rangle 
	\right|_t\right] \\
	&+\sup_{\alpha\in (0,1)} E\left[\left|\sum_{j=1}^\infty e^{-\lambda_j\cdot}G_j(0)\cdot 
	\langle S_i^{(\alpha)},S_j\rangle\right|_t\right]\\
	&\nquad\le  \sum_{j=1}^\infty  \sup_{\alpha\in (0,1)} \langle S_i^{(\alpha)},S_j\rangle\cdot 
	E\left[\left|\int_{u=0}^\cdot\lambda_je^{-\lambda_ju}W_j\left(e^{2\lambda_ju}-1\right)\, du  
	\right|_t\right]+\sum_{j=1}^\infty  \sup_{\alpha\in (0,1)}\langle S_i^{(\alpha)},S_j\rangle 
	\cdot |G_j(0)| \\
	&\nquad =\sum_{j=1}^\infty  \sup_{\alpha\in (0,1)} \langle S_i^{(\alpha)},S_j\rangle\cdot 
	\left( E\left[\int_{u=0}^t \lambda_je^{-\lambda_ju}\left|W_j\left(e^{2\lambda_ju}-1\right)  
	\right|\, du \right]+ |G_j(0)| \right)\\
	&\nquad\le\sum_{j=1}^\infty \sup_{\alpha\in (0,1)} \langle S_i^{(\alpha)},S_j\rangle\cdot 
	\left(c t\lambda_j+|G_j(0)|\right)\\
	&\nquad=\sum_{j=1}^\infty\sup_{\alpha\in (0,1)} \int_0^1\left(-g_i^{(\alpha)}\right)'(u)\cdot 
	S_j(u)\, du\cdot\left(c t\lambda_j+ |G_j(0)|\right) \\
	&\nquad \le \sup_{\alpha\in (0,1)}\left\|\left(g_i^{(\alpha)}\right)'(u)\right\|_{L^1}\cdot 
	\left\| \sum_{j=1}^\infty \left(c t\lambda_j+|G_j(0)|\right)\cdot S_j\right\|_{C}
\end{align*}
for some $c>0$ independent of $j$ and $t$. The claim now follows from
\begin{align*}
	\left\| \sum_{j=2}^\infty \left(c t\lambda_j+|G_j(0)|\right)\cdot S_j\right\|_{C}
	\le (cT\vee 1)\sum_{m=0}^\infty\left(\lambda_{d2^{m+1}}+\max_{d2^{m}<j\le d2^{m+1}}|G_j(0)| 
	\right)\left\|\sum_{j=d2^m+1}^{d2^{m+1}} S_j\right\|_{C}
\end{align*}
which is finite by $\left\|\sum_{j=d2^m+1}^{d2^{m+1}} S_j\right\|_{C}=2^{-\frac{m}{2}}$, 
\eqref{eq:rc:q1}, and \eqref{eq:rc:q2}.
\\(c) Let $i\in \N$. Furthermore let $\eps>0$ and $n\in \N$ such that 
\begin{align*}
	\sum_{j=n+1}^\infty |G_j(\lambda_jt)|\cdot\|S_j\|_C<\frac{\eps}{2\|\mu_i\|_v}
\end{align*}
for all $t\in [0,T]$. For the existence of such a random $n\in \N$ see \eqref{rc:lem:2:eq:1}. 
We have
\begin{align*}
	&\sum_{j=n+1}^\infty G_j(\lambda_jt)\left(\langle S^{(\alpha)}_i,S_j\rangle -\langle S_i,S_j 
	\rangle\right)=\sum_{j=n+1}^\infty G_j(\lambda_jt)\left(-\int_0^1 S_j\, d\mu_i^{(\alpha)} 
	+\int_0^1 S_j\, d\mu_i\right)\\
	&\quad \le \sum_{j=n+1}^\infty |G_j(\lambda_jt)|\cdot\|S_j\|_{C}\cdot\left(\|\mu_i^{(\alpha)} 
	\|_v+\|\mu_i\|_v\right)\le 2\sum_{j=n+1}^\infty |G_j(\lambda_jt)|\cdot\|S_j\|_{C}\cdot\|\mu_i 
	\|_v<\eps, 
\end{align*} 
the second last inequality because of $\|\mu_i^{(\alpha)}\|_v\le \|\mu_i\|_v$, $\alpha\in (0,1)$. 
Furthermore, since $\mu_i^{(\alpha)}$ converges to $\mu_i$ in the weak$^\ast$-topology as 
$\alpha\to 0$, 
\begin{align*}
	&\left| \langle S^{(\alpha)}_i,X_t\rangle -\langle S_i,X_t\rangle\right|
	\le \eps+\left|\sum_{j=1}^n G_j(\lambda_jt)\cdot \left(\langle S^{(\alpha)}_i,S_j\rangle 
	-\langle S_i,S_j\rangle\right)\right|\\
	&\quad\le \eps+\sum_{j=1}^n \left|G_j(\lambda_jt)\right|\cdot \left|-\int_0^1 S_j\, d 
	\mu_i^{(\alpha)}-\delta_{ij}\right|\\
	&\stack{\alpha\to 0}{\lra} \eps+\sum_{j=1}^n \left|G_j(\lambda_jt)\right|\cdot \left| 
	-\int_0^1 S_j\, d\mu_i-\delta_{ij} \right|=\eps,\quad i\in \N. 
\end{align*}
The claim follows.
\end{proof}

\begin{proposition}
\label{it:prop:1}
For $X_t$ given by \eqref{rc:intro:eq:1} and $F(s,X_s)=f(s\, ;\langle S_1,X_s 
\rangle,\ldots, \langle S_k,X_s\rangle)$, $f\in C_0^\infty\left(\R^{k+1}\right)$, $s\ge 0$, 
the following It\^o formula holds.
\begin{align*}
	F(t,X_t)&=F(0,X_0)+\int_0^t \frac{\partial}{\partial s} F(s,X_s)\, ds+\sum_{i=1}^k 
		\int_0^t \frac{\partial}{\partial x_i}f(s\, ;G_1(\lambda_1s),\ldots, G_k(\lambda_ks)) 
		\, d_sG_i(\lambda_is)\\
		&\quad+ \sum_{i=1}^k\int_0^t \lambda_i \frac{\partial^2}{\partial x_i^2} 
		f(s\, ;G_1(\lambda_1s),\ldots,G_k(\lambda_ks))\, ds.
\end{align*}
\end{proposition}
\begin{proof}
Next we apply the results of \cite{KurtzProtter1991} Section 2. Let 
$Y_n$, $n\in\N$, be a sequence of $\R$-valued semimartingales admitting a decomposition $Y_n=M_n+V_n$ such that 
for each $t\ge 0$ it holds that $\sup_{s\le t} |Y_n(s)-\eta(s)|\lni 0$ in probability, $\sup_n 
E[M_n(t)^2]=\sup_n E[[M_n]_t]<\infty$, as well as $\sup_n E[|V_n|_t]<\infty$. Under these conditions 
if $\sup_{s\le t}|X_n(s)-\xi(s)|\lni 0$ in probability then for each $T>0$
\begin{align} 
	\label{it:prop:1:eq:0}
	\sup_{t\le T}\left|\int_0^t X_n(s)\, dY_n(s)-\int_0^t \xi(s)\, d\eta(s)\right|\lni 0
\end{align}
in probability. As already mentioned, the above conclusions follow from \cite{KurtzProtter1991} Section 2.
From \eqref{it:eq:1} it follows that
\begin{align}
	\notag
	&\int_0^t {}_{B^\ast}\langle DF^{(\alpha)}(s,X_s),dX_s\rangle_B
	=\lim_{\eps\to 0}\int_0^t {\vphantom{\frac{X_{s+\eps}-X_s}{\eps}}}_{B^\ast}\left\langle 
	DF^{(\alpha)}(s,X_s),\frac{X_{s+\eps}-X_s}{\eps}\right\rangle_B\, ds\\
	\notag
	&\quad=\sum_{i=1}^k \lim_{\eps\to 0} \int_0^t \frac{\partial}{\partial x_i}f\left(s\, ; 
	\langle S^{(\alpha)}_1,X_s\rangle,\ldots,\langle S^{(\alpha)}_k,X_s\rangle\right)\cdot 
	\left(\frac{\langle S_i^{(\alpha)},X_{s+\eps}\rangle-\langle S_i^{(\alpha)},X_s\rangle}{\eps} 
	\right)\, ds\\
	\label{it:prop:1:eq:1}
	&\quad=\sum_{i=1}^k \int_0^t \frac{\partial}{\partial x_i}f\left(s\, ;\langle 
	S^{(\alpha)}_1, X_s\rangle,\ldots,\langle S^{(\alpha)}_k,X_s\rangle\right)\cdot 
	d\langle S_i^{(\alpha)},X_{s}\rangle
\end{align}
where the right hand side is an It\^o integral since $\langle S_i^{(\alpha)},X_s\rangle= 
\sum_{j=1}^\infty \langle S_i^{(\alpha)},S_j\rangle G_j(\lambda_js)$ is a semi-martingale, see 
Proposition 6 of \cite{RussoVallois2007}. From Lemma \ref{it:lem:1} it follows that the 
right-hand side converges in the sense of \eqref{it:prop:1:eq:0} to 
\begin{align}
	\label{it:prop:1:eq:3}
	\sum_{i=1}^k \int_0^t \frac{\partial}{\partial x_i}f\left(s\, ;\langle 
	S_1, X_s\rangle,\ldots,\langle S_k,X_s\rangle\right)\cdot 
	d\langle S_i,X_{s}\rangle\quad \text{as }\alpha\to 0.
\end{align}
We recall that 
\begin{align*}
	\Theta=2E\left[\left(\sum_{i=1}^\infty\lambda_i^\hf \xi_i S_i \right)\otimes\left(\sum_{i'= 
	1}^\infty \lambda_{i'}^\hf\xi_{i'} S_{i'} \right)\right]
\end{align*}
where $\xi_i$, $i\in \N$, are independent standard normal random variables. Using 
\cite{JansonKaijser2015}, Theorem 10.2, together with \eqref{eq:rc:q2} we obtain from 
\eqref{it:eq:2} 
\begin{align}
	\notag
	&\hf \int_0^t {\vphantom{\big(}}_{(B\ptp B)^\ast}\left\langle D^2F^{(\alpha)}(s,X_s),\Theta 
	\right\rangle_{(B\ptp B)^{\ast\ast}}\, ds \\
	\notag
	&\quad=\hf\sum_{l,l'=1}^k\int_0^t \frac{\partial^2}{\partial x_lx_{l'}}f\left(s\, ; 
	\langle S^{(\alpha)}_1, X_s\rangle,\ldots,\langle S^{(\alpha)}_k,X_s\rangle\right) \\
	\notag
	&\qquad\times \int_{[0,1]^2} \left(g_l^{(\alpha)}(u)\right)'\otimes \left(g_{l'}^{(\alpha)} 
	(v)\right)' \sbt (i\circ \Theta)(u,v)\, dudv\, ds \\	
	\notag
	&\quad=\sum_{l,l'=1}^k (\lambda_l\lambda_{l'})^\hf  \int_0^t\frac{\partial^2}{\partial x_l 
	x_{l'}}f\left(s\, ;\langle S^{(\alpha)}_1, X_s\rangle,\ldots,\langle S^{(\alpha)}_k,X_s 
	\rangle\right) \\
	\label{it:prop:1:eq:2}
	&\qquad\times \int_{[0,1]^2} \left(g_l^{(\alpha)}(u)\right)'\otimes \left(g_{l'}^{(\alpha)} 
	(v)\right)' \sbt \sum_{i=1}^\infty S_i(u)\otimes S_i(v)\,dudv\, ds.
\end{align}
Here the right-hand side converges almost surely to
\begin{align}
	\notag
	&\sum_{l,l'=1}^k (\lambda_l\lambda_{l'})^\hf  \int_0^t\frac{\partial^2}{\partial x_l 
	x_{l'}}f\left(s\, ;\langle S_1, X_s\rangle,\ldots,\langle S_k,X_s 
	\rangle\right)\\
	\notag
	&\qquad\times\int_{[0,1]^2}  \sum_{i=1}^\infty S_i(u)\otimes S_i(v) \sbt\,d\left(\mu_l(u)\otimes 
	\mu_{l'}(v)\right) \, ds \\
	\label{it:prop:1:eq:4}
	&\quad=\sum_{i=1}^k\lambda_i \int_0^t\frac{\partial^2}{\partial x_i^2}f\left(s\, ; 
	\langle S_1, X_s\rangle,\ldots,\langle S_k,X_s \rangle\right) \, ds\quad\text{as }\alpha\to 0
\end{align}
by Lemma \ref{it:lem:1} (c) and the hypothesis that $\mu_i^{(\alpha)}$ converges to $\mu_i$ 
in the weak$^\ast$-topology. The proposition now follows from the It\^o formula of Theorem 
\ref{it:theorem:1} for $F^{(\alpha)}(s,\gamma)=f(s\, ;\langle S^{(\alpha)}_1,\gamma_s\rangle, 
\ldots,\langle S^{(\alpha)}_k,\gamma_s\rangle)$ and \eqref{it:prop:1:eq:1}-\eqref{it:prop:1:eq:4}.
\end{proof} 

\begin{remark}
Reviewing the finite dimensional approximation of $X$ in Section \ref{rc:section:1} 
this particular specification of the It\^o formula provides a certain double-check of 
the correctness of Proposition \ref{tv:prop:1} and Theorem \ref{it:theorem:1} (b). 
\end{remark}







\appendix
\section{Appendix: Some lemmas in extreme value theory} 
\label{ap:section:1}

Let $G_i$, $i\in \N$, be a sequence of one-dimensional Ornstein-Uhlenbeck processes of the form 
\begin{align*}	
	G_i(t)=G_i(0)e^{-t}+e^{-t}W_i(e^{2t}-1),\quad t\ge 0, 
\end{align*}
with independent one-dimensional standard Wiener processes $W_i$. Let $1\le\lambda_1\le \lambda_2 
\le\ldots$ be a sequence of real constants. For $T>0$ denote 
$G_{m,T}^\ast:=\max \{G_i(\lambda_it)-G_i(0)e^{-\lambda_it}:t\in [0,T], \ d2^m< i\le d2^{m+1}\}$, $m\in {\mathbb Z}_+$. 
We are interested in certain moments of $G_{m,T}^\ast$. In order to obtain these moments let us first consider one Ornstein-Uhlenbeck process $G$ and $\lambda\ge 1$. Below we will use the notation $G^\ast_T:= 
\max_{t\in [0,T]}(G(\lambda t)-G(0)e^{-\lambda t})$.
\begin{lemma}
	\label{lem:ap:ou:1}
	The cumulative distribution function $F_G(x):=P(G_T^\ast\le x)$, $x\in {\mathbb R}$, is 
tail-equivalent to a von Mises function $F$ with $\lim_{x\to\infty}(1-F_G(x))/(1-F(x))=1$.
\end{lemma}
\begin{proof}
Let us use the representation $G(\lambda t)-G(0)e^{-\lambda t}=e^{-\lambda t}W(e^{2\lambda t}-1)$, 
$t\in[0,T]$, where $W$ is a suitable one-dimensional standard Wiener process. We have
\begin{align*}
	P(G^\ast_T\le x)
	&=P\big(\{G(t)-G(0)e^{-\lambda t}\le x \ \mbox{\rm for all}\ t\in [0,T]\}\big)\\
	&=P\big(\{W(s)\le x(s+1)^\hf\ \mbox{\rm for all}\ s\in [0,S]\}\big),
\end{align*}
where $S\equiv S(\lambda)=e^{2\lambda T}-1$. Let $\varphi$ and $\Phi$ denote the density and 
cumulative distribution function of the $N(0,1)$-distribution. Also let $\bar{\Phi}=1-\Phi$. 
Using the Corollary to Lemma 11 of \cite{Cuzick1981} we obtain
\begin{align}
	\notag
	&\nquad\nquad P(G_T^\ast>x)
	=P\big(\{W(s)>x(s+1)^\hf\ \text{ for some }s\in [0,S]\}\big)\\
	\label{eq:qv:int:1}
	&\nquad\sim\int_0^S \frac{1}{2s}\cdot \frac{x(s+1)^\hf}{s^\hf}\cdot\varphi\bigg(\frac 
    {x(s+1)^\hf}{s^\hf}\bigg)\, ds+\bar{\Phi}\bigg(\frac{x(S+1)^\hf}{S^\hf}\bigg)=:\bar{F}(x) 
	\equiv 1-F(x)\quad\mbox{\rm as } x\to\infty 
\end{align}
in the sense that the ratio of the left-hand side of $\sim$ and the right-hand side of $\sim$ 
tends to one as $x\to\infty$. We also mention that in (\ref{eq:qv:int:1}) the variable $x$ 
plays the role of $n$ in the Corollary to Lemma 11 of \cite{Cuzick1981} and that the conditions 
of this corollary follow immediately. 

As shown in \cite{EmKlMi}, Proposition 3.3.28 and Example 3.3.23, it is now sufficient to verify 
\begin{align}
	\label{ap:F:eq:1}
	\lim_{x\to \infty}\frac{\big(1-F(x)\big)F''(x)}{\big(F'(x)\big)^2}=-1
\end{align}
and $F''(x)<0$ for sufficiently large $x$. We note that
\begin{align*}
    \bar{\Phi}\bigg(\frac{x(S+1)^\hf}{S^\hf}\bigg)=\int_0^S \bigg(\frac{1}{2s}\cdot 
	\frac{x(s+1)^\hf}{s^\hf}-\frac{x}{2s^\hf (s+1)^\hf}\bigg)\cdot\varphi\bigg(\frac{x(s+1)^\hf}
	{s^\hf}\bigg)\, ds 
\end{align*}
which gives with \eqref{eq:qv:int:1} 
\begin{align*}
	\bar{F}(x)=\int_0^S \frac{1}{s}\cdot \frac{x(s+1)^\hf}{s^\hf}\cdot\varphi\bigg( 
	\frac{x(s+1)^\hf}{s^\hf}\bigg)\, ds-\int_0^S \frac{x}{2s^\hf (s+1)^\hf}\cdot\varphi 
	\bigg(\frac{x(s+1)^\hf}{s^\hf}\bigg)\, ds.
\end{align*}
We observe that with $\psi(s,x):=\varphi\left(x(s+1)^\hf\cdot s^{-\hf}\right)$ and $\mu(ds):= 
	\left((s+2)\cdot (2s)^{-\hf}(s+1)^{-\frac32}\right)\, ds$
we have 
\begin{align} 
	\label{eq:lem:ap:ou:1:1}
	\bar{F}(x)=-\int_0^S \frac{d}{dx} \bigg[\varphi\bigg(\frac{x(s+1)^\hf}{s^\hf}\bigg) 
	\bigg]\, \frac{(s+2)\, ds}{2s^\hf(s+1)^{\frac32}}\equiv x\cdot\int_0^S\psi(s,x)\cdot
	\frac{s+1}{s}\, \mu(ds)
\end{align} 
which implies 
\begin{align} 
	\label{eq:lem:ap:ou:1:2}
	\bar{F}'(x)=\int_0^S\psi(s,x)\cdot\frac{s+1}{s}\, \mu(ds)-x^2\cdot \int_0^S\psi(s,x)
\cdot\left(\frac{s+1}{s}\right)^2\, \mu(ds)
\end{align} 
and 
\begin{align} 
	\label{eq:lem:ap:ou:1:3}
	\bar{F}''(x)=-3x\cdot\int_0^S\psi(s,x)\cdot\left(\frac{s+1}{s}\right)^2\, \mu(ds) 
	+x^3\cdot \int_0^S\psi(s,x)\cdot\left(\frac{s+1}{s}\right)^3\, \mu(ds). 
\end{align} 
Since, for $k,l\in {0,1,2}$, 
\begin{align} 
	\label{eq:lem:ap:ou:1:4}
	\left.\int_0^S\psi(s,x)\cdot\left(\frac{s+1}{s}\right)^k\, \mu(ds)\right/ 
	\int_0^S\psi(s,x)\cdot\left(\frac{s+1}{s}\right)^l\, \mu(ds)\stack{x\to\infty}{\lra} 
	\left(\frac{S+1}{S}\right)^{k-l} 
\end{align} 
we get
\begin{align*}
	\lim_{x\to \infty} \frac{\bar{F}(x)\bar{F}''(x)}{\big(\bar{F}'(x)\big)^2}=1 
\end{align*}
which is \eqref{ap:F:eq:1}. From \eqref{eq:lem:ap:ou:1:3} we deduce $F''(x)<0$ for sufficiently 
large $x$. The statement follows.
\end{proof}

\begin{lemma}
	\label{lem:ap:ou:2}
(a) There exist sequences $c_n>0$ and $d_n\in \R$, $n\in \N$, such that 
\begin{align*}
	\lim_{\nti} F_G^n(c_nx+d_n)=\lim_{\nti} F^n(c_nx+d_n)=e^{-e^{-x}},\quad x\in \R,
\end{align*}
i.e. $F_G$ belongs to the domain of attraction of the Gumbel distribution.\\
(b) The sequences $c_n$ and $d_n$, $n\in \N$, can be chosen by 
\begin{align*}
	d_n:=F^{-1}\left(1-\frac{1}{n}\right)\quad\mbox{\rm and}\quad
	c_n:=\frac{\bar{F}(d_n)}{F'(d_n)}, 
\end{align*}
where $F^{-1}$ denotes the inverse of the restriction of $F$ to $[G(0),\infty)$. \\ 
(c) There exist the limits 
\begin{align}
	\label{eq:lem:ap:ou:2:1}
	c:=\lim_\nti c_n (\ln n)^\hf >0 
\end{align} 
and 
\begin{align} 
	\label{eq:lem:ap:ou:2:2}
	d:=\lim_\nti d_n (\ln n)^{-\hf} >0.
\end{align}
\end{lemma}
\begin{proof}
	(a) This is a consequence of Lemma \ref{lem:ap:ou:1} and \cite{EmKlMi} Theorem 3.3.26 as well 
	as Proposition 3.3.28.\\
	(b) This follows from \cite{EmKlMi} Theorem 3.3.26 and \cite{EmKlMi} Example 3.3.23.\\
	(c) Relation \eqref{eq:lem:ap:ou:2:2} is, on the one hand, derived from
	\begin{align}
	\label{eq:lem:ap:ou:2:3}
		-\ln \bar{F}(d_n)=\ln n 
	\end{align}
which follows from the definition of $d_n$. On the other hand, according to equation 
\eqref{eq:lem:ap:ou:1:1} and the mean value theorem we have 
	\begin{align} 
	\label{eq:lem:ap:ou:2:4}
	\bar{F}(x)=-\frac{d}{dx}\bigg[\varphi\bigg(\frac{x(s+1)^\hf}{s^\hf}\bigg)\bigg]
	\bigg|_{s=S_x}\cdot \int_0^S \, \frac{(s+2)\, ds}{2s^\hf(s+1)^{\frac32}}\equiv 
	\bar{\F}(x;S)\cdot C
	\end{align}
for some $S_x\in [0,S]$ with $\lim_{x\to\infty}S_x=S$. Relation \eqref{eq:lem:ap:ou:2:1} follows 
now from \eqref{eq:lem:ap:ou:1:1}, \eqref{eq:lem:ap:ou:1:2}, and \eqref{eq:lem:ap:ou:1:4}.
\end{proof}

In Lemma \ref{lem:ap:ou:1} and Lemma \ref{lem:ap:ou:2} we have analyzed the cumulative 
distribution function $F_G(x):=P(G_T^\ast\le x)$, $x\in {\mathbb R}$. Recall that $G^\ast_T= 
\max_{t\in [0,T]}(G(\lambda t)-G(0)e^{-\lambda t})$. The next lemma focuses 
on the impact of $\lambda\ge 1$ on the sequences $c_n\equiv c_n(\lambda)>0$ and $d_n\equiv 
d_n(\lambda)\in \R$, $n\in \N$, defined in Lemma \ref{lem:ap:ou:2}. 
\begin{lemma}
\label{lem:ap:ou:3} We have 
\begin{align} 
    \label{eq:lem:ap:ou:3:-1}
	c_n(\lambda)\ge c_0\big((\ln n)^\hf+(\ln \lambda )^\hf\big)^{-1}
\end{align}
and 
\begin{align} 
    \label{eq:lem:ap:ou:3:0}
	d_n(\lambda)\le d_0\big((\ln n)^\hf+(\ln \lambda )^\hf\big)
\end{align}
for some $c_0>0$ and $d_0>0$ independent of $n\ge 2$ and $\lambda\ge 1$.
\end{lemma}
\begin{proof}
	According to \eqref{eq:lem:ap:ou:2:4}
it holds for all $\lambda\ge 1$ that $\bar{F}(x)\equiv\bar{\F}(x;S)\cdot C(\lambda)$ where 
\begin{align}
	\label{eq:lem:ap:ou:3:1}
	C(\lambda)= \int_0^S \, \frac{(s+2)\, ds}{2s^\hf(s+1)^{\frac32}}\quad\text{with}\ 
	S\equiv S(\lambda)=e^{2\lambda T}-1.
\end{align}
From \eqref{eq:lem:ap:ou:2:3} we obtain 
	\begin{align} 
	    \label{eq:lem:ap:ou:3:3}
		-\ln (C(\lambda)\cdot \bar{\F}(d_n;S))=\ln n,
	\end{align}
and \eqref{eq:lem:ap:ou:2:4} gives 
	\begin{align*}
		\lim_{x\to\infty}\frac{\ln (\bar{\F}(x;S))}{x^2}=-\frac{S+1}{2S}. 
	\end{align*}
Taking into consideration \eqref{eq:lem:ap:ou:2:2}, using \eqref{eq:lem:ap:ou:3:3} it turns out 
that
	\begin{align*}
		d_n^2 \le d_1(\ln n+\ln C(\lambda))
	\end{align*}
for some $d_1>0$ independent of $n\ge 2$ and $\lambda\ge 1$. Relation \eqref{eq:lem:ap:ou:3:0} 
follows now from \eqref{eq:lem:ap:ou:3:1}. Together with Lemma \ref{lem:ap:ou:2} (b) and 
\eqref{eq:lem:ap:ou:1:1}, \eqref{eq:lem:ap:ou:1:2}, \eqref{eq:lem:ap:ou:1:4} this implies 
\eqref{eq:lem:ap:ou:3:-1}. 
\end{proof}

Now we turn to the main object of interest of this appendix, the estimates of the particular moments of  
$G_{m,T}^\ast:=\max \{G_i(\lambda_it)-G_i(0)e^{-\lambda_it}:t\in [0,T], \ d2^m< i\le d2^{m+1}\}$. 
Below we will use the notation $M^{(n)}\equiv M^{(n)}(\lambda):=\max\{(G_i(\lambda t)-G_i(0)e^{-\lambda t}):t\in [0,T],\ 
1\le i\le n\}$. Furthermore, let $\Gamma^{(k)}(1)$ denote the $k$th derivative of the gamma 
function at $x=1$. 
\begin{lemma} 
	\label{lem:rc:max:ou}
Let $c_n(\lambda)>0$ and $d_n(\lambda)\in \R$ 
be given by Lemma \ref{lem:ap:ou:2} (b) and the 
paragraph before Lemma \ref{lem:ap:ou:3}. \\ 
(a) For $k\in \N$ there is an $n_0\equiv n_{0,k}\in\N$ and a constant $C\equiv C_k>0$, both 
independent of $\lambda\ge 1$, such that 
\begin{align*}
	E\left[\left(M^{(n)}(\lambda)\right)^k\right]\le C\, (d_n(\lambda))^k,\quad n\ge n_0. 
\end{align*}
(b) For each $k_0\in N$ there is a constant $D\equiv D_{k_0}\in \R>0$ independent of $m\in\{1,2,\ldots\}$ 
such that for all $1\le k\le k_0$
\begin{align*}
	E\left[\left(G_{m,T}^\ast\right)^k\right]\le D\left((\ln \lambda_{d2^{m+1}})^{\frac{k}{2}} 
	+m^{\frac{k}{2}}\right).
\end{align*}
\end{lemma}
\begin{proof} 
(a) By \eqref{eq:qv:int:1} there is some finite $a\ge 1$ independent of $\lambda\ge 1$ such that 
\begin{align*}
	1-F_G(x)\le a(1-F(x)), \quad x>0,  
\end{align*}
where we mention that 
the functions $F_G\equiv F_{G,\lambda}$ and $F\equiv F_\lambda$ depend on $\lambda$ via $S\equiv 
S(\lambda)=e^{2\lambda T}-1$. Particularly the existence of this $a\ge 1$ can be derived from 
$F_\lambda(x)/F_{G,\lambda}(x)\stack{x\to\infty}{\lra}1$ for all $\lambda\ge 1$ and $F_\lambda(x) 
\stack{\lambda\to\infty}{\lra}\infty$ for all $x>0$. 

From the final conclusion of the proof of Lemma 2.2 (a) in \cite{Resnick} we learn that, in the notation used here, $-\ln F_G(c_n(\lambda)x+d_n(\lambda))\sim 1-F_G(c_n(\lambda)x+d_n(\lambda))$ as $n\to 
\infty$ uniformly for $x>0$ and $\lambda\ge 1$, see Lemma \ref{lem:ap:ou:3}. As in the proof of 
Lemma 2.2 (a) in \cite{Resnick}, for given $\eps>0$, there exist now $n_1\in\N$ independent of 
$\lambda\ge 1$ such that for all $n\ge n_1$, $x>0$, and $\lambda\ge 1$
\begin{align*}
	&1-F^n_G(c_n(\lambda)x+d_n(\lambda))
	\le (1+\eps)n\left(1-F_G(c_n(\lambda)x+d_n(\lambda)) 
	\vphantom{l^1}\right)\\
	&\quad\le a(1+\eps)n\left(1-F(c_n(\lambda)x+d_n(\lambda))\vphantom{l^1} 
	\right). 
\end{align*}

Now we follow and modify the proof of Lemma 2.2 (a) in \cite{Resnick} relative to a cumulative 
distribution function $F$ (no longer $F_G$) that satisfies \eqref{eq:qv:int:1}. For given 
$\eps>0$ there is $n_2\in \N$ independent of $\lambda\ge 1$ such that, now in symbols of 
\cite{Resnick}, $|f'(t)|<\eps$ if $t\ge b_n$ for $n\ge n_2$. In our notation the latter would 
mean $\big|\left(\bar{F}_\lambda(x)/F_\lambda'(x)\right)'\big|<\eps$ if $x\ge d_n(\lambda)$ for 
$\lambda\ge 1$ 
and $n\ge n_1$. This holds because of \eqref{eq:lem:ap:ou:1:1}-\eqref{eq:lem:ap:ou:1:4} 
taking into consideration that the limit \eqref{eq:lem:ap:ou:1:4} is uniform in $\lambda\ge 1$ 
by $S\equiv S(\lambda)=e^{2\lambda T}-1$. Furthermore, we can reduce the degree of $(1+\eps)$ 
by two since, in symbols of \cite{Resnick} but our situation, we have $1-F(b_n)=n^{-1}$ and 
$c(x)=c$, $x>0$. We arrive at 
\begin{align*}
	n\left(1-F(c_n(\lambda)x+d_n(\lambda))\vphantom{l^1}\right)\le (1+\eps x)^{-1/\eps}\quad 
	x>0,\ \lambda\ge 1,\ n\ge n_3
\end{align*} 
for some $n_3\in \N$ independent of $\lambda\ge 1$. Thus for sufficiently small $\eps>0$ there 
is an $n_0\equiv n_{0,k}\in \N$ such that for all $k\in \N$ with $\int_0^\infty kx^{k-1}\left( 
(1+\eps)(1+\eps x)^{-1/\eps}\right)\, dx <\infty$ and all $n\ge n_0$ the following holds. We 
have  
\begin{align*}
	1-F_G^n(c_n(\lambda)x+d_n(\lambda))\le a(1+\varepsilon)(1+\varepsilon x)^{-1/\varepsilon},
	\quad x>0,
\end{align*}
independent of $\lambda\ge 1$. For $n\ge n_0$ we obtain 
\begin{align*} 
	&E\left[\left(M^{(n)}(\lambda)\right)^k\right]=\int_0^\infty kx^{k-1}P\left(M^{(n)} 
	(\lambda)>x \right)\, dx \\ 
	&\quad=\int_{-d_n(\lambda)/c_n(\lambda)}^\infty k\left(c_n(\lambda)x+d_n(\lambda)\vphantom 
	{l^1}\right)^{k-1}P\left(M^{(n)}>c_n(\lambda) x+d_n(\lambda)\right)\cdot c_n(\lambda)\, dx \\ 
	&\quad\le\int_{-d_n/c_n}^0k(c_nx+d_n)^{k-1}\cdot c_n\, dx+\int_0^\infty k(c_nx+d_n)^{k-1} 
	\left(1-F_G^n(c_nx+d_n)\right)\cdot c_n\, dx \\ 
	&\quad\le d_n^k+ a\int_0^\infty k(c_nx+d_n)^{k-1}\left((1+\varepsilon)(1+\varepsilon x)^{-1/ 
	\varepsilon}\right)\cdot c_n\, dx.
\end{align*} 
Part (a) of the lemma follows. \\
(b) This is a consequence of Lemma \ref{lem:ap:ou:3}, and part (a) of the present lemma.
\end{proof}


\bibliography{bibfile}
\bibliographystyle{plain}

\end{document}